\documentclass[12pt]{amsart}

\usepackage{amsfonts}
\usepackage{amsmath}
\usepackage{amssymb}
\usepackage{mathrsfs}
\usepackage{amscd}
\usepackage[all]{xy}
\usepackage[pdftex,final]{graphicx}
\usepackage{hyperref}

\usepackage{exscale,txfonts}

\newtheorem{lem}{Lemma}[section]
\newtheorem{thm}[lem]{Theorem}
\newtheorem{prop}[lem]{Proposition}
\newtheorem{cor}[lem]{Corollary}
\theoremstyle{definition}
\newtheorem{defi}[lem]{Definition}
\newtheorem{exa}[lem]{Example}
\newtheorem{rem}[lem]{Remark}

\renewcommand{\iff}{\Longleftrightarrow}

\newcommand{\Q}{\Bbb{Q}}
\newcommand{\F}[1]{\Bbb{F}_{#1}}

\newcommand{\Z}{\Bbb{Z}}

\newcommand{\R}{\Bbb{R}}

\newcommand{\K}{\mathcal{K}}

\newcommand{\one}{\mathbb{D}}
\newcommand{\spl}[2]{\mathrm{SL}_{#1}(#2)}

\newcommand{\gl}[2]{\mathrm{GL}_{#1}(#2)}

\newcommand{\pb}[1]{\mathcal{P}(#1)}%  Pre-Bloch Group
\newcommand{\redpb}[1]{\widehat{\mathcal{P}}(#1)}
\newcommand{\rredpb}[1]{\overline{\mathcal{P}}(#1)}
\newcommand{\rpb}[1]{{\mathcal{RP}}(#1)}%  refined pre-Bloch group
\newcommand{\rrpb}[1]{\widetilde{\mathcal{RP}}(#1)}% reduced refined pre-Bloch group
\newcommand{\rrrpb}[1]{\widehat{\mathcal{RP}}(#1)}%further reduced refined Bloch group
\newcommand{\rrrrpb}[1]{\overline{\mathcal{RP}}(#1)}
\newcommand{\bl}[1]{\mathcal{B}(#1)}% Bloch group

\newcommand{\rbl}[1]{{\mathcal{RB}}(#1)}%refined Bloch group
\newcommand{\rrbl}[1]{\widetilde{\mathcal{RB}}(#1)}%reduced refined Bloch group
\newcommand{\rrrbl}[1]{\widehat{\mathcal{RB}}(#1)}%further reduced refined Bloch group
\newcommand{\rrrrbl}[1]{\overline{\mathcal{RB}}(#1)}
\newcommand{\rblker}[1]{\mathcal{RB}_0(#1)}

\newcommand{\gpb}[1]{\left[ #1\right]} %notation for generators of (pre)Bloch
\newcommand{\sus}[1]{\left\{ #1\right\} }% "Suslin element"
\newcommand{\suss}[2]{\psi_{#1}\left( #2\right)}

\newcommand{\ks}[2]{{\K}^{\mbox{\tiny $(#1)$}}_{#2}}
\newcommand{\kks}[1]{{\K}_{#1}}
\newcommand{\kl}[1]{\mathcal{S}_{#1}}
\newcommand{\bconst}[1]{C_{#1}}
\newcommand{\cconst}[1]{D_{#1}}
\newcommand{\cconstmod}[1]{\mathcal{D}_{#1}}

\newcommand{\rcr}{\mathrm{cr}}
\newcommand{\bw}{\Lambda}
\newcommand{\rbw}{\widetilde{\bw}}

\newcommand{\pn}[2]{\mathbb{P}^{#1}(#2)}
\newcommand{\projl}[1]{\pn{1}{#1}}

\newcommand{\sq}[1]{G_{#1}}

\newcommand{\card}[1]{\left| #1 \right|}

\newcommand{\id}[1]{\mathrm{Id}_{#1}}

\renewcommand{\ker}[1]{\mathrm{Ker}(#1)}

\newcommand{\coker}[1]{\mathrm{Coker}(#1)}

\renewcommand{\hom}[3]{\mathrm{Hom}_{#1}(#2,#3)}

\newcommand{\zero}{\{ 0 \}}

\newcommand{\modtwo}[1]{{#1}/2}

\newcommand{\ptor}[2]{{#1}_{(#2)}}

\newcommand{\torsion}[1]{\left(#1\right)_{\mbox{\tiny tors}}}

\newcommand{\sgr}[1]{\mathrm{R}_{#1}}% group ring of square classes of field
\newcommand{\an}[1]{\left\langle{#1}\right\rangle}
\newcommand{\pf}[1]{\left\langle\!\left\langle{#1}\right\rangle\!\right\rangle}
\newcommand{\aug}[1]{\mathcal{I}_{#1}}
\newcommand{\igr}[1]{\Z[#1]}

\newcommand{\pp}[1]{\mathrm{p}_{#1}^+}
\newcommand{\ppm}[1]{\mathrm{p}_{#1}^-}
\newcommand{\ep}[1]{\mathrm{e}_{#1}^+} % +-idempotent
\newcommand{\epm}[1]{\mathrm{e}_{#1}^-} % - idempotent
\newcommand{\idem}[1]{\mathrm{e}_{#1}}

\newcommand{\matr}[4]{\left[\begin{array}{cc}
#1&#2\\
#3&#4\\
\end{array}
\right]}

\newcommand{\sym}[3]{\mathrm{Sym}^{#1}_{#2}(#3)}

\newcommand{\asymm}{\circ}
\newcommand{\asym}[3]{\mathrm{S}^{#1}_{#2}(#3)}
\newcommand{\rasym}[3]{\mathrm{RS}^{#1}_{#2}(#3)}
\newcommand{\rrasym}[3]{\widetilde{\rasym{#1}{#2}{#3}}}
\newcommand{\rsf}[2]{[#1,#2]}

\newcommand{\zhalf}[1]{{#1}\mbox{$[\frac{1}{2}]$}}
\newcommand{\zzhalf}[1]{{#1}'}

\newcommand{\res}[1]{\arrowvert_{#1}}

\newcommand{\cl}[1]{\mathrm{Cl}(#1)}
\newcommand{\cls}[2]{\mathrm{Cl}^{#1}(#2)}

\newcommand{\milk}[2]{K^{\mathrm{\small M}}_{#1}({#2})}
\newcommand{\kind}[1]{K^{\mathrm{\small ind}}_3(#1)}
\newcommand{\ho}[3]{\mathrm{H}_{#1}(#2,#3 )}

\newcommand{\Tor}[2]{\mathrm{Tor}^{\Z}_{1}(#1,#2)}
\newcommand{\covtor}[1]{\mathrm{Tor}^{\Z}_1\widetilde{(\mu_{#1},\mu_{#1})}}

\title{A refined Bloch group and the third homology of $\mathrm{SL}_2$ of a field}
\author{Kevin Hutchinson}
\address{School of Mathematical Sciences,
 University College Dublin}
\email{kevin.hutchinson@ucd.ie}
\date{\today}


\begin{document}
\bibliographystyle{plain}
\maketitle

\begin{abstract}
 We use the properties of the refined Bloch group to study the structure of $\ho{3}{\spl{2}{F}}{\Z}$ for a field $F$.
We compute this group up to $2$-torsion when $F$ is a local field with finite residue field of odd order, 
and we show that 
for any global field $F$ it is not finitely-generated.
\end{abstract}
%\tableofcontents

\section{Introduction}\label{sec:intro}
In \cite{sah:discrete3}, Chih-Han Sah quotes S. Lichtenbaum who mentions our lack of knowledge of the 
precise structure of $\ho{3}{\spl{2}{\Q}}{\Z}$ as an example of the unsatisfactory state of our 
understanding of the homology of linear groups. This was nearly twenty five years ago, and to the author's knowledge 
the precise structure of this group is still unknown.  (Observe, however,
that $\ho{3}{\spl{n}{\Q}}{\Z}\cong\kind{\Q}\cong\Z/24$ for all $n\geq 3$ - 
by the results of \cite{hutchinson:tao2}, for example.) 

In this article, we study the structure of the third homology of $\mathrm{SL}_2$ of fields 
by using the properties of the \emph{refined Bloch group} of the field, which was introduced in \cite{hut:arxivcplx11}.
We are particularly interested in understanding $\ho{3}{\spl{2}{F}}{\Z}$ as a \emph{functor} of $F$, 
and its possible relation to other functors in $K$-theory and algebraic geometry.  
%The refined Bloch group captures  the (generally nontrivial) action of the 
%multiplicative group of the field on the third homology of the special linear group.  
%This work is part of a programme to try to understand 
%the homology of the special linear groups of fields below the stable range (see \cite{hutchinson:tao3}, for example).
 
What are now referred to as \emph{Bloch groups} of fields first appeared in the work of S. Bloch in the late 1970s 
(see \cite{bloch:regulators} for the 
lecture notes) as a way of 
constructing explicit maps - and, in particular, regulators - on $K_3$ of fields. 
In the 1980s, they were studied by Dupont and Sah 
(under the name \emph{scissors congruence group}) because of their connection with $3$-dimensional hyperbolic geometry 
(\cite{sah:dupont}, \cite{sah:discrete3}).  This connection 
is still actively studied today: Bloch groups of number fields are targets for \emph{Bloch invariants} of 
certain finite-volume oriented hyperbolic $3$-manifolds 
(\cite{neumann:yang}, \cite{neumann:zagier}, \cite{zickert:goette}). These 
invariants are amenable to explicit calculation and are related in a known way to the Chern-Simons invariant. 
There are also intriguing connections between Bloch groups, conformal field theories and even modular form theory 
(\cite{zagier:dilog},\cite{nahm:bloch}).

The precise relationship between the Bloch group and $K$-theory was established via their mutual connection to the 
homology of linear groups. 
These connections were greatly clarified and exploited in the work of Suslin (\cite{sus:bloch}): 
For a field $F$ with at least $4$ elements, the pre-Bloch group, $\pb{F}$, of the field is an abelian group with
 generators 
$\gpb{x}$, $x\in F^\times\setminus\{ 1\}$, subject to a family of $5$-term relations. The Bloch group, 
$\bl{F}$, 
 is a subgroup of $\pb{F}$ which also arises naturally as a quotient of $\ho{3}{\gl{2}{F}}{\Z}$. Suslin has shown 
that this extends to a surjection from   $\ho{3}{\gl{3}{F}}{\Z}$ to $\bl{F}$ (\cite[section 3]{sus:bloch}). 
Of course, $K_3(F)$ admits a Hurewicz homomorphism to 
$\ho{3}{\gl{}{F}}{\Z}=\ho{3}{\gl{3}{F}}{\Z}$, and hence, by composition, a map to $\bl{F}$. Using calculations of the 
 homology of $\gl{n}{F}$ ($n=2,3$) as well as the homotopy theory of the plus construction, 
Suslin proved that there is a natural short 
exact sequence
\[
0\to \widetilde{\Tor{\mu_F}{\mu_F}}\to\kind{F}\to \bl{F}\to 0
\]
where $ \widetilde{\Tor{\mu_F}{\mu_F}}$ is the unique nontrivial extension of $\Tor{\mu_F}{\mu_F}$ by $\Z/2$ and 
\[
\kind{F}=\coker{K_3^M(F)\to K_3(F)}.
\]

These results of Suslin were extended to finite fields (with at least $4$ elements) in \cite[section 7]{hut:arxivcplx11}. 

In a letter to Sah (see \cite[section 4]{sah:discrete3}), Suslin asked the question whether the composite \\
$\ho{3}{\spl{2}{F}}{\Z}\to \ho{3}{\spl{}{F}}{\Z}\to\kind{F}$ 
induces an isomorphism 
\[
\ho{3}{\spl{2}{F}}{\Z}_{F^\times}\cong\kind{F}.
\]
The current state of knowledge on this question is that the map is surjective 
(\cite{hutchinson:tao2}) and, if we let $\zhalf{A}$ denote $A\otimes\zhalf{\Z}$ for any abelian group $A$,
that the induced map
\[
\ho{3}{\spl{2}{F}}{\zhalf{\Z}}_{F^\times}\to\zhalf{\kind{F}}
\]
is an isomorphism  (\cite{mirzaii:third}). 

The homology groups of the special linear groups $\spl{n}{F}$ are naturally modules over the 
group ring $\Z[F^\times]$ via 
the short exact sequences 
\[
\xymatrix{
1\ar[r]
&\spl{n}{F}\ar[r]
&\gl{n}{F}\ar[r]^-{\det}
&F^\times\ar[r]
&1.
}
\]
Since the scalar matrices $a\cdot I_n$ are central and have determinant $a^n$, it 
follows that $(F^\times)^n$ acts trivially 
on $\ho{k}{\spl{n}{F}}{\Z}$. In particular, the groups $\ho{k}{\spl{2}{F}}{\Z}$ are modules for the group ring 
$\sgr{F}:=\Z[F^\times/(F^\times)^2]$.

When $n>k$ (or $n\geq k$ when $k$ is odd) in the groups $\ho{k}{\spl{n}{F}}{\Z}$, 
we are in the range of stability (see \cite{sah:discrete3}, 
\cite{hutchinson:tao2} and \cite{hutchinson:tao3}) and 
this module structure is necessarily trivial. But below the range 
of stability, the module structure appears to be nontrivial and interesting. For example, the unstable groups 
$\ho{2n}{\spl{2n}{F}}{\Z}$ are modules over the Grothendieck-Witt ring of the field (which is a quotient of 
$\sgr{F}$) and surject onto the even Milnor-Witt $K$-theory groups of the field $F$ (\cite{hutchinson:tao3}).

In \cite{hut:arxivcplx11} the author defined a \emph{refined pre-Bloch group}, $\rpb{F}$, of a field $F$ and a subgroup, 
 the \emph{refined Bloch group}, $\rbl{F}$,
which was shown to have the following properties:  
\begin{enumerate}
\item\cite[Theorem 4.3]{hut:arxivcplx11}
 The group $\rbl{F}$ is an $\sgr{F}$-module and there is a natural surjective homomorphism of $\sgr{F}$-modules 
\[
\xymatrix{
\ho{3}{\spl{2}{F}}{\Z}\ar@{>>}[r]
&\rbl{F}
}
\]

\item 
This induces a commutative diagram (of $\zhalf{\sgr{F}}$-modules) with exact rows
\begin{eqnarray*}
\xymatrix{
0\ar[r]
&\zhalf{\Tor{\mu_F}{\mu_F}}\ar[r]\ar[d]^-{=}
&\ho{3}{\spl{2}{F}}{\zhalf{\Z}}\ar[r]\ar@{>>}[d]
&\zhalf{\rbl{F}}\ar[r]\ar@{>>}[d]
&0\\
0\ar[r]
&
\zhalf{\Tor{\mu_F}{\mu_F}}\ar[r]
&\zhalf{\kind{F}}\ar[r]
&\zhalf{\bl{F}}\ar[r]
&0
}
\end{eqnarray*}
Here $\sgr{F}$ acts trivially on the bottom row and furthermore
\item\cite[Corollary 5.1]{hut:arxivcplx11}
On taking $F^\times$-coinvariants, the top row becomes isomorphic to the bottom row. In particular,
$\zhalf{\rbl{F}}_{F^\times}\cong\zhalf{\bl{F}}$ and 
\[
\aug{F}\zhalf{\rbl{F}}\cong\aug{F}\ho{3}{\spl{2}{F}}{\zhalf{\Z}}=\ker{\ho{3}{\spl{2}{F}}{\zhalf{\Z}}\to\zhalf{\kind{F}}}
\]
where $\aug{F}$ denotes the augmentation ideal of the group ring $\sgr{F}$.

\item\cite[Lemma 5.2]{hut:arxivcplx11} 
The group $\aug{F}\zhalf{\rbl{F}}$ is also the kernel of the stabilization homomorphism 
\[
\ho{3}{\spl{2}{F}}{\zhalf{\Z}}\to\ho{3}{\spl{3}{F}}{\zhalf{\Z}}.
\]
 
The cokernel of this map is the third Milnor $K$-group $\zhalf{\milk{3}{F}}$ and the image is isomorphic 
to $\zhalf{\kind{F}}$ (\cite[Theorem 4.7, Section 5]{hutchinson:tao2}). 
\end{enumerate}

The main result of the current article (Theorem \ref{thm:spec} and its corollaries)
tells us that given a valuation $v$ on the field $F$ with residue field $k$, there 
are surjective reduction or specialization 
homomorphisms 
\[
\xymatrix{
\rpb{F}\ar@{>>}[r] 
&\rrrpb{k}
}
\]
where $\rrrpb{k}$ is a certain quotient of $\rpb{k}$.

In particular, if $a\in F^\times$ and $v(a)$ is not a multiple of $2$, there is a specialization homomorphism which induces 
a surjection
\[
\xymatrix{
\epm{a}\zhalf{\rbl{F}}\ar@{>>}[r]
&\zhalf{\redpb{k}}
}
\]
where 
\[
\epm{a}=\frac{1-\an{a}}{2}\in\zhalf{\aug{F}}
\]
and $\an{a}$ denotes the square class of $a$ in $\sgr{F}$ and $\redpb{k}$ is a certain quotient of the classical pre-Bloch 
group, $\pb{k}$, of the residue field.

Using these results, we can prove that $\aug{F}\zhalf{\rbl{F}}$ is large if $F$ is a field with many valuations. 
In particular, 
it follows that $\ho{3}{\spl{2}{F}}{\Z}$ is not finitely generated for any global field $F$. 
By contrast, if $F$ is a global 
field then $\ho{3}{\spl{3}{F}}{\Z}=\kind{F}$ is well-known to be finitely generated.
More precisely, if $F$ is a global field whose  class group has odd order then there is 
a natural surjection
\[
\xymatrix{
\aug{F}\zhalf{\rbl{F}}\ar@{>>}[r]
&
\bigoplus_{v}\zhalf{\pb{k_v}}
}
\] 
where $v$ runs through all the finite places of $F$ and $k_v$ is the residue field at $v$ (see Corollary 
\ref{exa:global}). Let $\zzhalf{n}$ denote the 
odd part of the nonzero integer $n$. By the results of 
\cite{hut:arxivcplx11}, if $\F{q}$ is the finite field with $q$ elements,  then $\zhalf{\pb{\F{q}}}$ is finite cyclic 
of order $\zzhalf{(q+1)}$.

As another application, 
we also use the techniques developed to construct explicitly  non-trivial $\F{3}$-vector spaces of known dimension 
inside groups of the form $\ho{3}{\spl{2}{\mathcal{O}_S}}{\Z}$ where $\mathcal{O}_S$ is a  ring of $S$-integers 
corresponding to a set $S$ of primes of the field $F$, 
and thus to give lower bounds on the $3$-ranks of such 
groups. Since there is still very little in the literature at present by way of explicit calculations of the 
homology or cohomology of $S$-arithmetic groups of this type, we hope that the techniques developed here will be 
useful in proving more general results of this type in the future. In particular, our results point to the importance
of the $\mathcal{O}_S^\times$-module structure in calculations of this type. 

Finally, we use these specialization maps, 
together with the basic algebraic properties of the refined Bloch groups, developed in section \ref{sec:rpf} below, 
to give a calculation of $\ho{3}{\spl{2}{F}}{\zhalf{\Z}}$ for local fields $F$ with finite residue field $k$
of odd order (Theorem \ref{thm:local}). In particular, for such fields (with some minor restrictions if $\Q_3\subset F$) 
we have 
\[
\ho{3}{\spl{2}{F}}{\zhalf{\Z}}\cong \zhalf{\kind{F}}\oplus \zhalf{\pb{k}}.
\] 
Here the right-hand side is an $\sgr{F}$ module with $F^\times$ acting trivially on the 
first factor while any uniformizer acts as $-1$ on the second factor.

To return to our opening remarks, it follows from the results here that there is a natural surjection
\[
\xymatrix{
\ho{3}{\spl{2}{\Q}}{{\Z}}\otimes\zhalf{\Z}=\ho{3}{\spl{2}{\Q}}{\zhalf{\Z}}\ar@{>>}[r]
&
\zhalf{\kind{\Q}}\oplus\left(\bigoplus_{p}\zhalf{\pb{\F{p}}}\right).
}
\]
It is natural to ask whether this is an isomorphism and, furthermore, what adjustments, if any, need to be made 
to obtain a corresponding statement with integral coefficients.

\textbf{Acknowledgements.} The author thanks the anonymous referee for a very thorough reading of the original manuscript 
and for numerous detailed suggestions which have  greatly improved the exposition. 
Any remaining obscurities are entirely the responsibility of the author. 

\section{Review of Bloch Groups}\label{sec:bloch}

\subsection{Preliminaries and Notation}

For a field $F$, we let $\sq{F}$ denote the multiplicative group, $F^\times/(F^\times)^2$, of nonzero 
square classes of the field.
For $x\in F^\times$, we will let $\an{x} \in \sq{F}$ denote the corresponding square class. 
Let $\sgr{F}$ denote the integral group ring $\igr{\sq{F}}$ of the group $\sq{F}$. 
We will use the notation $\pf{x}$ for the basis elements, $\an{x}-1$, 
of the augmentation ideal $\aug{F}$ of $\sgr{F}$.

For any $a\in F^\times$, we will let $\pp{a}$ and $\ppm{a}$ denote respectively the elements $1+\an{a}$ and $1-\an{a}$ 
in $\sgr{F}$.

We let $\ep{a}$ and $\epm{a}$ denote 
respectively the mutually orthogonal idempotents
\[
\ep{a}:=\frac{\pp{a}}{2}=\frac{1+\an{a}}{2},\quad \epm{a}:=\frac{\ppm{a}}{2}=\frac{1-\an{a}}{2}\in\zhalf{\sgr{F}}. 
\] 
(Of course, these operators depend only on the class of $a$ in $\sq{F}$.)
 
For any abelian group $A$ we will let $\zhalf{A}$ denote $A\otimes\zhalf{\Z}$. For an integer $n$, we will let 
$\zzhalf{n}$ denote the odd part of $n$. Thus if $A$ is a finite abelian group of order $n$, then $\zhalf{A}$ is 
a finite abelian group of order $\zzhalf{n}$. For an additive abelian group $A$, $\modtwo{A}$ denotes 
$A\otimes\modtwo{\Z}$. However, for multiplicative groups $M$ we will use the notation $M/M^2$. If $n$ is a positive 
integer and $A$ an abelian group, $\ptor{A}{n}$ will denote the set $\{ a\in A | na=0\}$.

\subsection{The classical Bloch group}

For a field $F$, with at least $4$ elements, the \emph{pre-Bloch group}, $\pb{F}$, is the group 
generated 
by the elements $\gpb{x}$, $x\in F^\times\setminus\{ 1\}$,  subject to the relations 
\[
R_{x,y}:\quad  \gpb{x}-\gpb{y}+\gpb{y/x}-\gpb{(1-x^{-1})/(1-y^{-1})}+\gpb{(1-x)/(1-y)} \quad x\not=y.
\] 

Let $\asym{2}{\Z}{F^\times}$ denote the group 
\[
\frac{F^\times\otimes_{\Z}F^\times}{<x\otimes y + y\otimes x | x,y \in F^\times>}
\]
and denote by $x\asymm y$ the image of $x\otimes y$ in $\asym{2}{\Z}{F^\times}$.

The map 
\[
\lambda:\pb{F}\to \asym{2}{\Z}{F^\times},\quad  [x]\mapsto \left(1-{x}\right)\asymm {x}
\]
is well-defined, and the \emph{Bloch group of $F$}, $\bl{F}\subset \pb{F}$, is defined to be the kernel of $\lambda$.

\subsection{The refined Bloch group}

The \emph{refined pre-Bloch group}, $\rpb{F}$, of a field $F$ which has at least $4$ elements,
 is the $\sgr{F}$-module with generators $\gpb{x}$, $x\in F^\times$ 
subject to the relations $\gpb{1}=0$ and 
\[
S_{x,y}:\quad 0=[x]-[y]+\an{x}\left[ y/x\right]-\an{x^{-1}-1}\left[(1-x^{-1})/(1-y^{-1})\right]
+\an{1-x}\left[(1-x)/(1-y)\right],\quad x,y\not= 1
\]

Of course, from the definition it follows immediately that $\pb{F}=(\rpb{F})_{F^\times}=\ho{0}{F^\times}{\rpb{F}}$.

For any field $F$ we define the $\sgr{F}$-module
\begin{eqnarray*}
\rasym{2}{\Z}{F^\times}:=& \aug{F}^2\times_{\sym{2}{\F{2}}{\sq{F}}}\asym{2}{\Z}{F^\times}\subset 
\aug{F}^2\oplus \asym{2}{\Z}{F^\times}
\end{eqnarray*}
where $\asym{2}{\Z}{F^\times}$ has the trivial $\sgr{F}$-module structure. 

As an $\sgr{F}$-module, $\rasym{2}{\Z}{F^\times}$ is generated by the elements 
\[
\rsf{a}{b}:= (\pf{a}\pf{b},a\asymm b)\in \rasym{2}{\Z}{F^\times}. 
\]

The \emph{refined Bloch-Wigner homomorphism} $\bw$ to be the $\sgr{F}$-module homomorphism
\[
\bw :\rpb{F}\to\rasym{2}{\Z}{F},\qquad \gpb{x}\mapsto \rsf{1-x}{x}
\]
which is well-defined by \cite[proof of Theorem 4.3]{hut:arxivcplx11}.

In view of the definition of $\rasym{2}{\Z}{F^\times}$, we can express $\bw= (\lambda_1,\lambda_2)$ where 
$\lambda_1:\rpb{F}\to \aug{F}^2$ is the map $\gpb{x}\mapsto \pf{1-x}\pf{x}$, and $\lambda_2$ is the composite
\[
\xymatrix{
\rpb{F}\ar@{>>}[r]
&\pb{F}\ar[r]^-{\lambda}
&\asym{2}{\Z}{F^\times}.
}
\] 

Finally, we can define the 
\emph{refined Bloch group} of the field $F$ (with at least $4$ elements) to be the $\sgr{F}$-module 
\[
\rbl{F}:=\ker{\bw: \rpb{F}\to \rasym{2}{\Z}{F^\times}}.
\]

\subsection{The fields $\F{2}$ and $\F{3}$}
Throughout this paper it will be convenient for us to have (refined and classical) pre-Bloch and Bloch groups 
for the fields 
with $2$ and $3$ elements. For this reason, we introduce the following \textit{ad hoc} definitions. 

$\pb{\F{2}}=\rpb{\F{2}}=\rbl{\F{2}}=\bl{\F{2}}$ 
is simply an additive group of order $3$ with distinguished generator, denoted $\bconst{\F{2}}$.

$\rpb{\F{3}}$ is the cyclic $\sgr{\F{3}}$-module generated by the symbol $\gpb{-1}$ and subject to the one relation
\[
0=2\cdot(\gpb{-1}+\an{-1}\gpb{-1}).
\] 
The homomorphism
\[
\bw:\rpb{\F{3}}\to\rasym{2}{\Z}{\F{3}^\times}=\aug{\F{3}}^2=2\cdot \Z\pf{-1}
\]
is the $\sgr{\F{3}}$-homomorphism sending $\gpb{-1}$ to $\pf{-1}^2=-2\pf{-1}$. 

Then $\rbl{\F{3}}=\ker{\bw}$ is the submodule of order $2$ generated by $\gpb{-1}+\an{-1}\gpb{-1}$.

Furthermore, we let $\pb{\F{3}}=\rpb{\F{3}}_{\F{3}^\times}$. This is a cyclic $\Z$-module of order $4$ with generator 
$\gpb{-1}$. Let $\lambda:\pb{\F{3}}\to \asym{2}{\Z}{\F{3}^\times}=\sym{2}{\F{2}}{\F{3}^\times}$ be the map 
$\gpb{-1}\mapsto -1\asymm -1$. Then 
$\bl{\F{3}}:=\ker{\lambda}$ is cyclc of order $2$ with generator $2\gpb{-1}$ and the natural map 
$\rbl{\F{3}}\to \bl{\F{3}}$ is an isomorphism.  

\subsection{The refined Bloch Group and $\ho{3}{\spl{2}{F}}{\Z}$}

We recall some results from \cite{hut:arxivcplx11}: The main result there is 

\begin{thm}\label{thm:main} Let $F$ be an infinite field.

There is a natural complex 
\[
0\to \Tor{\mu_F}{\mu_F}\to\ho{3}{\spl{2}{F}}{{\Z}}\to{\rbl{F}}\to 0.
\]
which is exact everywhere except possibly at the middle term. The middle homology is annihilated by $4$.

In particular, for any infinite field there is a natural short exact sequence
\[
0\to\zhalf{\Tor{\mu_F}{\mu_F}}\to\ho{3}{\spl{2}{F}}{\zhalf{\Z}}\to\zhalf{\rbl{F}}\to 0.
\]
\end{thm}

The following result is Corollary 5.1 in \cite{hut:arxivcplx11}:

\begin{lem}
\label{lem:blochinf}
Let $F$ be an infinite field. Then the natural map $\rbl{F}\to\bl{F}$ is surjective and the induced map 
$\rbl{F}_{F^\times}\to \bl{F}$ has a $2$-primary torsion kernel.
\end{lem}

Now for any field $F$, let 
\begin{eqnarray*}
\ho{3}{\spl{2}{F}}{\Z}_0:=\ker{\ho{3}{\spl{2}{F}}{\Z}\to\kind{F}}
\end{eqnarray*}
and
\begin{eqnarray*}
\rblker{F}:=\ker{\rbl{F}\to\bl{F}}
\end{eqnarray*}

\begin{lem}\cite[Lemma 5.2]{hut:arxivcplx11} \label{lem:h3sl20}
Let $F$ be an infinite field. Then
\begin{enumerate}
\item $\ho{3}{\spl{2}{F}}{\zhalf{\Z}}_0=\zhalf{\rbl{F}}_0$
\item $\ho{3}{\spl{2}{F}}{\zhalf{\Z}}_0=\aug{F}\ho{3}{\spl{2}{F}}{\zhalf{\Z}}$ and
$\zhalf{\rblker{F}}=\aug{F}\zhalf{\rbl{F}}$.
\item 
$\ho{3}{\spl{2}{F}}{\zhalf{\Z}}_0 = \ker{\ho{3}{\spl{2}{F}}{\zhalf{\Z}}\to \ho{3}{\spl{3}{F}}{\zhalf{\Z}}}\\
= \ker{\ho{3}{\spl{2}{F}}{\zhalf{\Z}}\to \ho{3}{\gl{2}{F}}{\zhalf{\Z}}}$.
\end{enumerate}
\end{lem}

On the other hand, the corresponding results for finite fields are as follows (the results in \cite{hut:arxivcplx11} 
apply to fields with at least $4$ elements, but it is straightforward to verify that they extend to the fields 
$\F{2}$ and $\F{3}$ with the definitions supplied above): 

\begin{lem}\cite[Lemma 7.1]{hut:arxivcplx11}. 
\label{lem:blochfin}
For a finite field $k$ the natural map $\rpb{k}\to\pb{k}$ induces an isomorphism $\rbl{k}\cong\bl{k}$.
\end{lem}

For a field $F$, we let $\covtor{F}$ denote 
the unique nontrivial extension of $\Tor{\mu_F}{\mu_F}$ by $\Z/2$ if the characteristic 
of $F$ is not $2$, and   $\Tor{\mu_F}{\mu_F}$ in characteristic $2$.

\begin{thm}\cite[Corollary 7.5]{hut:arxivcplx11} \label{thm:blochfinite} 
There is a natural short exact sequence
\[
0\to\covtor{\F{q}}\to\ho{3}{\spl{2}{\F{q}}}{\Z[1/p]}\to\bl{\F{q}}\to 0
\]
for any finite field $\F{q}$ of order $q=p^f$.

Furthermore, there is a natural isomorphism 
\[
\ho{3}{\spl{2}{\F{q}}}{\Z[1/p]}\cong\kind{\F{q}}.
\]
\end{thm}

\begin{lem}\cite[Section 5,7]{hut:arxivcplx11}\label{lem:arxivcplx}
\[
\bl{\F{q}}\cong \left\{
\begin{array}{ll}
\Z/(q+1)/2,& q\mbox{ odd}\\
\Z/(q+1),& q \mbox{ even}
\end{array}
\right.
\]
and if $K\subset \spl{2}{\F{q}}$ is a cyclic subgroup of order $\zzhalf{(q+1)}$ then the composite map 
\[
\Z/\zzhalf{(q+1)}\cong\ho{3}{K}{\zhalf{\Z}}\to\ho{3}{\spl{2}{\F{q}}}{\zhalf{\Z}}\to \zhalf{\bl{\F{q}}}
\]
is an isomorphism.
\end{lem}

\begin{cor} For any prime power $q$, $\zhalf{\pb{\F{q}}}$ is cyclic of order $\zzhalf{(q+1)}$.
\end{cor}
\begin{proof} $\asym{2}{\Z}{\F{q}^\times}$ has order dividing $2$, so the inclusion $\bl{\F{q}}\to\pb{\F{q}}$ induces 
an isomorphism $\zhalf{\bl{\F{q}}}\cong \zhalf{\pb{\F{q}}}$.
\end{proof}

\subsection{The map $\ho{3}{G}{\Z}\to\rbl{F}$}  
If $G$ is any subgroup of $\spl{2}{F}$, then by composing the map $\ho{3}{\spl{2}{F}}{\Z}\to \rbl{F}$ 
of Theorem \ref{thm:main}  with the map $\ho{3}{G}{\Z}\to \ho{3}{\spl{2}{F}}{\Z}$ induced by the inclusion of 
$G$ into $\spl{2}{F}$
we obtain a map $\ho{3}{G}{\Z}\to\rbl{F}$.
In \cite[Section 6]{hut:arxivcplx11} a recipe is given for 
calculating this map for subgroups $G$ of $\spl{2}{F}$ which don't act transitively on $\projl{F}$. 
We recall this 
calculation in the case that $G$ is a  finite cyclic subgroup.

First, given $4$ distinct points $x_0,x_1,x_2,x_3\in\projl{F}$ we define the \emph{refined cross ratio}  
$\rcr(x_0,x_1,x_2,x_3)\in\rpb{F}$ by 
\begin{eqnarray*}
\rcr(x_0,x_1,x_2,x_3)=\left\{
\begin{array}{ll}
\an{\frac{(x_2-x_0)(x_0-x_1)}{x_2-x_1}}\gpb{\frac{(x_2-x_1)(x_3-x_0)}{(x_2-x_0)(x_3-x_1)}},
&\mbox{ if } x_i\not=\infty\\
&\\
\an{x_1-x_2}\gpb{\frac{x_1-x_2}{x_1-x_3}},&\mbox{ if } x_0=\infty\\
&\\
\an{x_2-x_0}\gpb{\frac{x_3-x_0}{x_2-x_0}},&\mbox{ if } x_1=\infty\\
&\\
\an{x_0-x_1}\gpb{\frac{x_3-x_0}{x_3-x_1}},&\mbox{ if } x_2=\infty\\
&\\
\an{\frac{(x_2-x_0)(x_0-x_1)}{x_2-x_1}}\gpb{\frac{x_2-x_1}{x_2-x_0}},&\mbox{ if } x_3=\infty\\
\end{array}
\right.
\end{eqnarray*}

\begin{lem}\cite[Section 6]{hut:arxivcplx11}\label{lem:h3g2rbf}
Suppose that $G$ is a finite cyclic subgroup of $\spl{2}{F}$ of order $r$ with generator $t$. Let
$x\in \projl{F}$ with trivial stabilizer $G_x=1$, and let $y\in \projl{F}\setminus G\cdot x$. 

The composite 
\[
\Z/r\cong \ho{3}{G}{\Z} \to \rbl{F}
\]
is given by the formula 
\[
1\mapsto \sum_{i=0}^{r-1}\rcr(\beta_3^{x,y}(1,t,t^{i+1},t^{i+2})).
\]
where
\begin{eqnarray*}
\beta_3^{x,y}(1,t,t,t^2)&=&0\\
\beta_3^{x,y}(1,t,t^{i+1},t^{i+2})&=&(x,t(x),t^{i+1}(x),t^{i+2}(x))\mbox{ for }1\leq i\leq r-3\\
\beta_3^{x,y}(1,t,t^{r-1},1)&=&(y,t(x),t^{-1}(x),x)-(y, x, t(x),t^{-1}(x))\\
\beta_3^{x,y}(1,t,t^r,t^{r+1})&=&\beta_3^{x,y}(1,t,1,t)\\
&=&\left\{
\begin{array}{ll}
0, & y=t(y)\\
(y,t(y),x,t(x))+(y,t(y),t(x),x),& y\not= t(y)
\end{array}
\right. 
\end{eqnarray*}

Furthermore, the resulting map is independent 
of the particular choice of $x$ and $y$.
\end{lem}

\section{Some algebra  in $\rpb{F}$}\label{sec:rpf}

In this section we study certain key elements and submodules of the refined pre-Bloch group of a field 
$F$.

\subsection{The elements $\suss{i}{x}$ and the modules $\ks{i}{F}$}

We recall the elements   
\[
\sus{x}:=
\gpb{x}+\gpb{x^{-1}}\in\pb{F}
\]
(for $x\in F^\times$). A straightforward calculation - see Suslin \cite{sus:bloch} - 
shows that these symbols allow us to define a group 
homomorphism
\[
F^\times\to \pb{F},\quad x\mapsto \sus{x}
\]
whose kernel contains $(F^\times)^2$; i.e. we have 
\[
\sus{x^2}=0\mbox{ and} \sus{xy}=\sus{x}+\sus{y}\mbox{ for all } x,y. 
\]
In particular, these elements satisfy $2\sus{x}=0$ for all $x$.

We now consider two liftings of these elements in $\rpb{F}$: For $x\in F^\times$ we let
\[
\suss{1}{x}:=\gpb{x}+\an{-1}{\gpb{x^{-1}}}
\]
and 
\[
\suss{2}{x}:=\left\{
\begin{array}{ll}
\an{1-x}\left(\an{x}\gpb{x}+\gpb{x^{-1}}\right),& x\not= 1\\
0,& x=1
\end{array}
\right.
\]

(If $F=\F{2}$, we interpret this as $\suss{i}{1}=0$ for $i=1,2$. For $F=\F{3}$, we have $\suss{1}{-1}=\suss{2}{-1}
=\gpb{-1}+\an{-1}\gpb{-1}$. )

The maps $F^\times\to\rpb{F}, x\mapsto \suss{i}{x}$ are no longer homomorphisms, but are $1$-cocycles 
for the action of $F^\times$:

\begin{lem}\label{lem:deriv}
Let $F$ be a field. 
For $i\in\{1,2\}$,  the map 
\[
F^\times\to\rpb{F}, x\mapsto \suss{i}{x}
\]
 is a $1$-cocycle; i.e. we have 
\[
\suss{i}{xy}=\an{x}\suss{i}{y}+\suss{i}{x}\mbox{ for all } x,y\in F^\times.
\]
\end{lem}
\begin{proof} The statement is trivial for $F=\F{2}$ or $\F{3}$. We can thus assume $F$ has at least $4$ 
elements.
 
If $x=1$ or $y=1$, the required identities are clear.  If $x\not=1$ and $y\not=x^{-1}$ the relation
$0=S_{x,xy}+\an{-1}S_{x^{-1},x^{-1}y^{-1}}$ in $\rpb{F}$ yields the identity
\[
\suss{1}{x}-\suss{1}{xy}+\an{x}\suss{1}{y}=0.
\]
Thus we must also prove that $\an{x}\suss{1}{x^{-1}}+\suss{1}{x}=0$ for all $x\not=1$. Fix $x\not=1$ and choose 
$y\not\in\{ 1, x^{-1}\}$ (here we use that $F$ has at least $4$ elements). Then 
\[
\an{y}\suss{1}{x}=\suss{1}{xy}-\suss{1}{y}=-\an{xy}\suss{1}{x^{-1}}
\]
and multiplying by $\an{y}$ gives the required identity.

Now, for $x, y\in F^\times$, let 
\[
Q(x,y):=\an{x}\gpb{\frac{x}{y}}+\an{y}\gpb{\frac{y}{x}}\in\rpb{F}.
\]
Then
\begin{eqnarray*}
Q(x,y)=\an{y}\left(\an{\frac{x}{y}}\gpb{\frac{x}{y}}+\gpb{\frac{y}{x}}\right)
&=&\an{y}\an{1-\frac{x}{y}}\suss{2}{\frac{x}{y}}=\an{y-x}\suss{2}{\frac{x}{y}}.
\end{eqnarray*}

For $a,b\not=1$, the relation $0=S_{a,b}+S_{b,a}$ in $\rpb{F}$ gives the identity
\[
Q(a^{-1}-1,b^{-1}-1)=Q(a^{-1},b^{-1})+Q(1-a,1-b).
\]
Thus
\[
\an{b^{-1}-a^{-1}}\suss{2}{\frac{a^{-1}-1}{b^{-1}-1}}=\an{b^{-1}-a^{-1}}\suss{2}{\frac{b}{a}}+\an{a-b}
\suss{2}{\frac{1-a}{1-b}}
\] 
and hence
\[
\suss{2}{\frac{b^{-1}-1}{a^{-1}-1}}=\suss{2}{\frac{b}{a}}+\an{\frac{b}{a}}\suss{2}{\frac{1-a}{1-b}}.
\]

Now if we fix $x,y\not=1$ with $xy\not=1$, we can solve the equations
 \[
x=\frac{b}{a},\qquad y=\frac{1-a}{1-b}
\]
for $a$ and $b$ and prove the required identity for $\suss{2}{\mbox{\ }}$.
\end{proof}

Observe from the definitions that, for $i=1,2$, 
$\suss{i}{x^{-1}}=\an{-1}\suss{i}{x}$ for all $x\in F^\times$. In particular,
$\an{-1}\suss{i}{-1}=\suss{i}{-1}$ for $i=1,2$.  

The properties enumerated in following proposition will be used often in the remainder of this section.

\begin{prop}\label{prop:cocycle}
For $i\in \{1,2\}$ we have:
\begin{enumerate}
\item $\pf{x}\suss{i}{y}=\pf{y}\suss{i}{x}$ for all $x,y$
\item
 $\suss{i}{xy^2}=\suss{i}{x}+\suss{i}{y^2}$ for all $x,y$
\item $\pf{x}\suss{i}{y^2}=0$ for all $x,y$
\item $2\cdot\suss{i}{-1}=0$ for all $i$ %and $\suss{i}{-1}=0$ is $-1$ is a square in $F$.
\item $\suss{i}{x^2}=-\pf{x}\suss{i}{-1}$ for all $x$
\item $2\cdot\suss{i}{x^2}=0 $ for all $x$ and if $-1$ is a square in $F$ then $\suss{i}{x^2}=0$ for 
all $x$.
\item $\pf{x}\pf{y}\suss{i}{-1}=0$ for all $x,y$
\item $\an{-1}\pf{x}\suss{i}{y}=\pf{x}\suss{i}{y}$ for all $x,y$
\item Let 
\[
\epsilon(F):=\left\{
\begin{array}{ll}
1, & -1\in (F^\times)^2\\
2,&  -1\not\in (F^\times)^2
\end{array}
\right.
\]
The map $\sq{F}\to\rpb{F},\an{x}\mapsto \epsilon(F)\suss{i}{x}$ is a well-defined $1$-cocycle.
\end{enumerate}
\end{prop}
\begin{proof}
The identities $\suss{i}{1}=0$ and $\suss{i}{x^{-1}}=\an{-1}\suss{i}{x}$ follow from the definition of 
$\suss{i}{\mbox{\ }}$.  More generally, let $M$ be an $\sgr{F}$-module and 
let $\psi:F^\times\to M$ be a $1$-cocycle  satisfying
\[
\psi(1)=0\mbox{ and } \psi(x^{-1})=\an{-1}\psi(x)\mbox{ for all } x\in F^\times.
\] 
\begin{enumerate}
\item By the cocycle condition, for all $x,y\in F^\times$ we have 
\[
\psi(xy)=\an{x}\psi(y)+\psi(x)=\an{y}\psi(x)+\psi(y).
\]
Rearranging this latter equality, we deduce: $\pf{x}\psi(y)=\pf{y}\psi(x)$ for all $x,y$. 
\item For all $x,y$ we have 
\[
\psi(xy^2)=\an{y^2}\psi(x)+\psi(y^2)=\psi(x)+\psi(y^2)
\]
since $\an{y^2}=1$ in $\sgr{F}$.
\item From (1) together with the fact that $\pf{y^2}=0$ for all $y$, we deduce:
\[
\pf{x}\psi(y^2)=\pf{y^2}\psi(x)=0.
\] 
\item We have $\an{-1}\psi(-1)=\psi(-1)$ and thus 
\[
0=\psi(1)=\psi(-1\cdot -1)=\psi(-1)+\an{-1}\psi(-1)=2\psi(-1).
\]
\item For all $x$ we have
\[
\psi(x)=\psi\left(\frac{1}{x}\cdot x^2\right)=\psi\left(\frac{1}{x}\right)+\psi(x^2)=
\an{-1}\psi(x)+\psi(x^2).
\]
Thus
\[
\psi(x^2)=-\pf{-1}\psi(x)=-\pf{x}\psi(-1).
\]
\item The first statement follows from (4) and (5). For the second, observe that for any $x$ we 
have 
\[
\psi(x^2)=-\pf{-1}\psi(x)
\]
and $\pf{-1}=0$ if $-1$ is a square.
\item This statement follows from (3) and (5).
\item This is a restatement of (7); namely
\[
\pf{-1}\pf{x}\psi(y)=\pf{x}\pf{y}\psi(-1)=0.
\]
\item By (6),  $\epsilon(F)\psi(x^2)=0$ in $M$ for all $x$ and thus $\epsilon(F)\psi(xy^2)=\epsilon(F)\psi(x)$ 
for all $x,y$. Thus the proposed map is well-defined (and is thus clearly a $1$-cocycle). 
\end{enumerate}

\end{proof}

Now, for $i\in \{ 1,2\}$,  let $\ks{i}{F}$ denote the $\sgr{F}$-submodule of $\rpb{F}$ generated by the set 
$\{ \suss{i}{x}\ |\ x\in F^\times\}$. 

\begin{lem} \label{lem:kf}
Let $F$ be a field.
 Then for $i\in \{ 1,2\}$
\[
\lambda_1\left( \ks{i}{F}\right) = \pp{-1}(\aug{F})\subset \aug{F}^2
%\lambda_2\left( \ks{F}\right)= -1\wedge F^\times\subset\extpow{2}{F^\times}
\]
and $\ker{\lambda_1\res{\ks{i}{F}}}$ is annihilated by $4$.
\end{lem}
\begin{proof}
We use the identities
\[
\pf{a}\pf{b}=\pf{ab}-\pf{a}-\pf{b},\quad \an{-1}\pf{a}=\pf{-a}-\pf{-1}, \quad \pf{ab^2}=\pf{a}
\]
in $\aug{F}$. 

Thus 
\begin{eqnarray*}
\lambda_1(\suss{1}{x})&=&\lambda_1(\gpb{x})+\an{-1}\lambda_1(\gpb{x^{-1}})\\
&=&\pf{x}\pf{1-x}+\an{-1}\pf{x}\pf{x(x-1)}\\
&=& \pf{x(1-x)}-\pf{x}-\pf{1-x}+\an{-1}(\pf{x-1}-\pf{x}-\pf{x(x-1)})\\
&=& \pf{x(1-x)}-\pf{x}-\pf{1-x}
+\pf{1-x}-\pf{-x}-\pf{x(1-x)}+\pf{-1}\\
&=&\pf{-1}-\pf{x}-\pf{-x}=\pf{-x}\cdot\pf{x}\\
&=&-\pp{-1}\cdot\pf{x}
\end{eqnarray*}

Thus $\lambda_1(\ks{1}{F})=\pp{-1}(\aug{F})$.

For $x\not=1$ we have $\suss{2}{x}=\an{x(1-x)}\gpb{x}+\an{1-x}\gpb{x^{-1}}$ and thus
\begin{eqnarray*}
\lambda_1(\suss{2}{x})&=&\an{x(1-x)}\pf{x}\pf{1-x}+\an{1-x}\pf{x}\pf{x(x-1)}\\
&=&\an{x(1-x)}(\pf{x(1-x)}-\pf{x}-\pf{1-x})+\an{1-x}(\pf{x-1}-\pf{x(x-1)}-\pf{x})\\
&=& -\pf{1-x}-\pf{x}+\pf{x(1-x)}+\pf{-1}-\pf{-x}-\pf{x(1-x)}+\pf{1-x}\\
&=&\pf{-1}-\pf{x}-\pf{-x}=\pf{x}\cdot\pf{-x}=-\pp{-1}\cdot\pf{x}.
\end{eqnarray*}

Thus $\lambda_1(\ks{2}{F})=\pp{-1}(\aug{F})$ also.
%On the other hand,
%\begin{eqnarray*}
%\lambda_2(\sus{x})&=&\left(\frac{x-1}{x}\right)\wedge\frac{1}{x}+(1-x)\wedge x\\
%&=& -(x-1)\wedge x + (1-x) \wedge x  = -1\wedge x.
%\end{eqnarray*}
%This proves the first statement of the Lemma.

For the second statement, recall that for any group $G$ and any $\Z[G]$-module $M$ a $1$-cocycle 
$\rho:G\to M$ gives rise to $\Z[G]$-homomorphism $\aug{G}\to M$, defined by $g-1\mapsto \rho(g)$. Thus, for $i\in\{1,2\}$, 
 we have a well-defined 
$\sgr{F}$-homomorphism 
\[
\aug{F}\to\rpb{F},\ \pf{x}\mapsto 2\suss{i}{x}.
\]
Combining this with the inclusion $\pp{-1}(\aug{F})\to\aug{F}$ we obtain an $\sgr{F}$-module homomorphism 
$\mu:\pp{-1}(\aug{F})\to \ks{i}{F}$ sending 
$\pp{-1}\pf{x}=\pf{x}+\an{-1}\pf{x}$ to $2\suss{i}{x}+\an{-1}2\suss{i}{x}$. However,  for any $x\in F^\times$ we 
have $2\pf{-1}\suss{i}{x}=2\pf{x}\suss{i}{-1}=0$ by Proposition \ref{prop:cocycle} (4), so that $2\an{-1}\suss{i}{x}=
2\suss{i}{x}$ and thus $2\suss{i}{x}+\an{-1}2\suss{i}{x} = 4\suss{i}{x}$. 

It follows that $\mu\circ \left({\lambda_1}\res{\ks{i}{F}}\right)$ is just multiplication by $4$, 
and the result is proved.
\end{proof}

\begin{rem}
Since $\pp{-1}\aug{F}$ is a free abelian group, it follows that, as an abelian group, 
$\ks{i}{F}$ decomposes as a direct sum $A\oplus \torsion{\ks{i}{F}}$ where
$A$ is a free abelian group and $4$ annihilates $\torsion{\ks{i}{F}}=\ker{\lambda_1\res{\ks{i}{F}}}$. 

Furthermore, if $\pf{-1}\suss{i}{x}=0$ for all $x$ (for example, if $-1\in (F^\times)^2$) then $\suss{i}{x^2}=0$ 
for all $x$ and the map $\sq{F}\to\rpb{F}, \an{x}\mapsto \suss{i}{x}$ is already a well-defined $1$-cocycle. 
The above arguments then show that $\torsion{\ks{i}{F}}=\ker{\lambda_1\res{\ks{i}{F}}}$ is annihilated by $2$. 
\end{rem}

\subsection{The constants $\bconst{F}$ and $\cconst{F}$}

In the classical pre-Bloch group  $\pb{F}$ the expression $\gpb{x}+\gpb{1-x}$ is known to be independent of 
$x\in F\setminus \{ 0,1\}$ (see, for example, \cite[Section 1]{sus:bloch}). 
Furthermore, this  constant has order dividing $6$. We consider now an analogous 
constant in $\rpb{F}$.

Let $F$ be a field with at least $4$ elements. For $x\in F^\times$, $x\not=1$ we let 
\[
\tilde{C}(x):=\gpb{x}+\an{-1}\gpb{1-x}\mbox{ and } C(x)=\tilde{C}(x)+\pf{1-x}\suss{1}{x}.
\]

\begin{lem} Let $F$ be a field with at least $4$ elements. Then 
$C(x)$ is constant; i.e. for all $x,y\in F\setminus\{ 0,1\}$ we have $C(x)=C(y)$ in 
$\rpb{F}$.
\end{lem}
\begin{proof}
In $\rpb{F}$ we have 
\begin{eqnarray*}
0&=& S_{x,y}+\an{-1}S_{1-x,1-y}\\
&=&\gpb{x}-\gpb{y}+\an{x}\gpb{\frac{y}{x}}-\an{x^{-1}-1}\gpb{\frac{x^{-1}-1}{y^{-1}-1}}+\an{1-x}\gpb{\frac{1-x}{1-y}}\\
&&+\an{-1}\gpb{1-x}-\an{-1}\gpb{1-y}+\an{x-1}\gpb{\frac{1-y}{1-x}}-\an{1-x^{-1}}\gpb{\frac{y^{-1}-1}{x^{-1}-1}}+
\an{-x}\gpb{\frac{x}{y}}\\
&=&\tilde{C}(x)-\tilde{C}(y)+\an{x}\suss{1}{\frac{y}{x}}-\an{1-x^{-1}}\suss{1}{\frac{1-y^{-1}}{1-x^{-1}}}+
\an{x-1}\suss{1}{\frac{y-1}{x-1}}\\
&=&\left(\tilde{C}(x)-\suss{1}{x}+\suss{1}{1-x^{-1}}-\suss{1}{x-1}\right)-
\left(\tilde{C}(y)-\suss{1}{y}+\suss{1}{1-y^{-1}}-\suss{1}{y-1}\right)
\end{eqnarray*}
(using the cocycle property of $\suss{1}{\mbox{\ }}$ to obtain the last line).
Furthermore 
\begin{eqnarray*}
\suss{1}{1-x^{-1}}-\suss{1}{x-1}-\suss{1}{x}&=&\suss{1}{(x-1)x^{-1}}-\suss{1}{x-1}-\suss{1}{x}\\
=\an{x-1}\suss{1}{x^{-1}}-\suss{1}{x}&=&\an{1-x}\suss{1}{x}-\suss{1}{x}=\pf{1-x}\suss{1}{x}.
\end{eqnarray*}
\end{proof}

\begin{defi} Thus, for a given field $F$ with at least $4$ elements, we will denote by 
$\bconst{F}$ the common value of the expressions $C(x)$ for $x\in F\setminus\{ 0,1\}$.

Furthermore, we let $\bconst{\F{2}}$ denote the distinguished generator of $\rpb{\F{2}}=\rbl{\F{2}}$ and 
we let $\bconst{\F{3}}:=\suss{1}{-1}=(1+\an{-1})\gpb{-1}\in\rbl{\F{3}}$. 
\end{defi}

For any field $F$, we let  $\Phi(T)$ denote the polynomial $T^2-T+1\in F[T]$. Observe 
that if $F$ has characteristic $3$ then $-1$ is a root of $\Phi(T)$. In any other characteristic, a root of $\Phi(T)$ will
be of the form $-\zeta$, where $\zeta$ is a primitive cube root of unity.

\begin{cor}\label{cor:bconst}
 Let $F$ be a field. 
\begin{enumerate}
\item $\an{-1}\bconst{F}=\bconst{F}$
\item If $\Phi(T)$ has a root in $F$, then $\bconst{F}=\suss{1}{-1}$. 
\item For any field $F$, $\bconst{F}\in \rbl{F}$; i.e $\bw(\bconst{F})=0$.
\end{enumerate}
\end{cor}
\begin{proof}
\begin{enumerate}
\item For $x\not\in\{0,1\}$ we have $\bconst{F}=C(1-x)=\an{-1}C(x)=\an{-1}\bconst{F}$ 
(since $\an{-1}\pf{1-x}\suss{1}{x}=\pf{1-x}\suss{1}{x}$ by Proposition \ref{prop:cocycle} (8)).
\item $\F{2}$ contains no root of $\Phi(T)$. If $F=\F{3}$, then $-1$ is a root of $\Phi(T)$, but 
the result is true by definition. So we can assume that 
$F$ has at least $4$ elements.

Let $x$ be a root of $\Phi(T)$. Then $1-x=x^{-1}$ and thus
\begin{eqnarray*}
\bconst{F}=C(x)&=&\gpb{x}+\an{-1}\gpb{1-x}+\pf{1-x}\suss{1}{x}\\
&=&\suss{1}{x}+\pf{x^{-1}}\suss{1}{x}\\
&=&\suss{1}{x}+\pf{x}\suss{1}{x}=\an{x}\suss{1}{x}.
\end{eqnarray*}

Now $x^3=-1$, so that $x=-1\cdot x^4$ and hence $\an{x}=\an{-1}$. 

Furthermore, using Proposition \ref{prop:cocycle} (2) and (6), it follows that
\begin{eqnarray*}
=\suss{1}{x}=\suss{1}{-1\cdot x^4}=\suss{1}{-1}+\suss{1}{x^4}=\suss{1}{-1}+2\suss{1}{x^2}=\suss{1}{-1}\\
\end{eqnarray*}

Thus $\bconst{F}=\an{-1}\suss{1}{-1}=\suss{1}{-1}$ as claimed.

\item Fix $x\in F^\times\setminus\{ 1\}$. Recall that $\lambda_1(\suss{1}{x})=-\pp{-1}\pf{x}$ (see the proof of 
Lemma \ref{lem:kf}). Thus 
\begin{eqnarray*}
\lambda_1(\bconst{F})&=& \lambda_1(\gpb{x})+\an{-1}\lambda_1(\gpb{1-x})-\pf{1-x}\lambda_1(\suss{1}{x})\\
&=& \pf{1-x}\pf{x}+\an{-1}\pf{x}\pf{1-x}-\pf{1-x}\pp{-1}\pf{x}\\
&=& \pp{-1}\pf{1-x}\pf{x}-\pp{-1}\pf{1-x}\pf{x}=0.
\end{eqnarray*}

On the other hand, recalling that $F^\times$ acts trivially on $\asym{2}{\Z}{F^\times}$, we have 
\begin{eqnarray*}
\lambda_2(\bconst{F})&=& \lambda_2(\gpb{x})+\an{-1}\lambda_2(\gpb{1-x})-\pf{1-x}\lambda_2(\suss{1}{x})\\
&=& (1-x)\asymm x + x\asymm (1-x)=0 \mbox{ in } \asym{2}{\Z}{F^\times}
\end{eqnarray*}
proving the result.
\end{enumerate}
\end{proof}

\begin{defi} For any field $F$, we define the element $\cconst{F}\in \rbl{F}$ by $\cconst{F}:=2\bconst{F}$. 
\end{defi}

Observe that $\cconst{F}=(1+\an{-1})\bconst{F}$ by  Corollary \ref{cor:bconst} (1).

\begin{lem}
For any $x\in F\setminus\{ 0,1\}$ we have 
\[
\cconst{F}=\gpb{x}+\an{-1}\gpb{\frac{1}{1-x^{-1}}}-\suss{1}{\frac{1}{1-x}}
\]
\end{lem}
\begin{proof}
\begin{eqnarray*}
\cconst{F}=2\bconst{F}&=&C(x)+C\left(\frac{1}{1-x}\right)\\
&=&\gpb{x}+\an{-1}\gpb{1-x}+\pf{1-x}\suss{1}{x}+\gpb{\frac{1}{1-x}}+\an{-1}\gpb{\frac{1}{1-x^{-1}}}+
\pf{1-x^{-1}}\suss{1}{\frac{1}{1-x}}.
\end{eqnarray*}

However, 
\[
\an{-1}\gpb{1-x}+\gpb{\frac{1}{1-x}}=\suss{1}{\frac{1}{1-x}}
\]
and  
\[
\pf{1-x}\suss{1}{x}=\pf{\frac{1}{1-x}}\suss{1}{x}=\pf{x}\suss{1}{\frac{1}{1-x}}.
\]

Thus 
\begin{eqnarray*}
\cconst{F} 
&=& \gpb{x}+\an{-1}\gpb{\frac{1}{1-x^{-1}}}+\left(1+\pf{x}+\pf{1-x^{-1}}\right)\suss{1}{\frac{1}{1-x}}\\
&=& \gpb{x}+\an{-1}\gpb{\frac{1}{1-x^{-1}}}+\left(\an{x}+\an{1-x^{-1}}-1\right)\suss{1}{\frac{1}{1-x}}.
\end{eqnarray*}

So it remains to prove that 
\[
\an{x}\suss{1}{\frac{1}{1-x}}=-\an{1-x^{-1}}\suss{1}{\frac{1}{1-x}}
\]
for any $x\not= 0,1$.

Now, by the cocycle property,
\[
\an{x}\suss{1}{\frac{1}{1-x}}=\suss{1}{\frac{x}{1-x}}-\suss{1}{x}.
\] 

However, for any $a\in F^\times$, we have $\suss{1}{a}=\an{-1}\suss{1}{-a}-\suss{1}{-1}$, and thus for any 
$a,b\in F^\times$ we have 
\[
\suss{1}{a}-\suss{1}{b}=\an{-1}\suss{1}{-a}-\an{-1}\suss{1}{-b}=\suss{1}{-a^{-1}}-\suss{1}{-b^{-1}}.
\]
It follows that 
\[
\an{x}\suss{1}{\frac{1}{1-x}}=\suss{1}{-\frac{1-x}{x}}-\suss{1}{-x^{-1}}=\suss{1}{1-x^{-1}}-\suss{1}{-x^{-1}}.
\]

On the other hand, using the cocycle property again,
\begin{eqnarray*}
\an{1-x^{-1}}\suss{1}{\frac{1}{1-x}}&=& \suss{1}{\frac{1-x^{-1}}{1-x}}-\suss{1}{1-x^{-1}}\\
&=& \suss{1}{-x^{-1}}-\suss{1}{1-x^{-1}}
\end{eqnarray*}
completing the proof.
\end{proof}

\begin{lem} 
For any field $F$  we have $3\cconst{F}=6 \bconst{F}=0$.
\end{lem}
\begin{proof}
Fix $x\not= 0,1$. Then 
\begin{eqnarray*}
3\cconst{F}&=& D\left(\frac{1}{x}\right)+D\left(\frac{1}{1-x^{-1}}\right)+D(1-x)\\
&=& \gpb{\frac{1}{x}}+\an{-1}\gpb{\frac{1}{1-x}}-\gpb{\frac{1}{1-x^{-1}}}-\an{-1}\gpb{1-x^{-1}}\\
&+&\gpb{\frac{1}{1-x^{-1}}}+\an{-1}\gpb{x}-\gpb{1-x}-\an{-1}\gpb{\frac{1}{1-x}}\\
&+&\gpb{1-x}+\an{-1}\gpb{1-x^{-1}}-\gpb{\frac{1}{x}}-\an{-1}\gpb{x}\\
&=&0.
\end{eqnarray*}
\end{proof}

\begin{rem} Examples show that this is best possible.  Under the natural map $\rpb{F}\to\pb{F}$ the image 
of $\bconst{F}$ is the constant $C:=\gpb{x}+\gpb{1-x}\in\bl{F}$. It can be shown that this element has order $6$ 
for example when $F=\R$ (\cite[Section 1]{sus:bloch}) or when $F$ is a finite field with $q$ elements and 
$q\equiv -1\pmod{12}$
(\cite[Lemma 7.11]{hut:arxivcplx11}).
\end{rem}

\begin{thm}\label{thm:df}
Let $F$ be a field. Then 
\begin{enumerate}
\item  For all $x\in F^\times$,  
\[
\pf{x}\cconst{F}=\suss{1}{x}-\suss{2}{x}.
\]
\item Let $E$ be the field extension obtained from $F$ by adjoining a root of $\Phi(T)$. 
Then 
\[
\pf{x}\cconst{F}=0\mbox{ if }\pm x\in N_{E/F}(E^\times)\subset F^\times.
\]
\end{enumerate}
\end{thm}
\begin{proof}
We consider first the case of a finite field $F$. The results in the final section of \cite{hut:arxivcplx11} show that 
the natural map $\rpb{F}\to\pb{F}$ induces an isomorphism $\rbl{F}\cong\bl{F}$. Now $\cconst{F}\in\rbl{F}$ and thus
$\pf{x}\cconst{F}=0$ for all $x$. Similarly for all $x$, $\suss{1}{x}-\suss{2}{x}\in\rbl{F}$ and this maps to
$\sus{x}-\sus{x}=0$ in $\bl{F}$. 

Thus, we can assume without loss that $F$ is an infinite field.
 
Let
\[
t:=\matr{-1}{1}{-1}{0}\in \spl{2}{F}.
\]
So 
\[
t^2=t^{-1}=\matr{0}{-1}{1}{-1}
\]
and for $x\in \projl{F}$ we have $t(x)=1-x^{-1}$ and $t^{-1}(x)=(1-x)^{-1}$. 

%We let $\Phi(T)$ be the polynomial $T^2-T+1\in F[T]$. 

Thus $t(x)=x$ if and only if $\Phi(x)=0$. 

We now choose $x\in F^\times\setminus\{ 1\}$ with $t(x)\not= x$ (i.e. $x$ not a root of $\Phi(T)$) 
and $y\in F^\times\setminus\{ 1\}$ satisfying 
$t(y)\not=y$ and $y\not\in \{ x,t(x),t^{-1}(x)\}$. 

By Lemma \ref{lem:h3g2rbf}, the natural composite map 
\[
\Z/3=\ho{3}{\an{t}}{\Z}\to\ho{3}{\spl{2}{F}}{\Z}\to\rbl{F}
\] 
is given by the formula
\[
1\mapsto \sum_{i=0}^{2}\rcr(\beta_3^{x,y}(1,t,t^{i+1},t^{i+2})):=C(x,y).
\]

Thus, using Lemma \ref{lem:h3g2rbf} again, we have 
\begin{eqnarray*}
C(x,y)&=&\rcr(\beta^{x,y}_3(1,t,t^{-1},1))+\rcr(\beta^{x,y}_3(1,t,1,t)\\
&=&\rcr(y,t(x),t^{-1}(x),x)-\rcr(y,x,t(x),t^{-1}(x))+\rcr(y,t(y),x,t(x))+\rcr(y,t(y),t(x),x)\\
&=& \an{\frac{(t^{-1}(x)-y)(y-t(x))}{t^{-1}(x)-t(x)}}\gpb{\frac{(t^{-1}(x)-t(x))(x-y)}{(t^{-1}(x)-y)(x-t(x))}}-
\an{\frac{(t(x)-y)(y-x)}{t(x)-x}}\gpb{\frac{(t(x)-x)(t^{-1}(x)-y)}{(t(x)-y)(t^{-1}(x)-x)}}\\
&&+\an{\frac{(x-y)(y-t(y))}{x-t(y)}}\gpb{\frac{(x-t(y))(t(x)-y)}
{(x-y)(t(x)-t(y))}}+\an{\frac{(t(x)-y)(y-t(y))}{t(x)-t(y)}}\gpb{\frac{(t(x)-t(y))(x-y)}
{(t(x)-y)(x-t(y))}}\\
\end{eqnarray*}

Now let 
\[
a=a(x,y)=\frac{(x-y)(t(x)-y)}{t(x)-x}.
\]

Multiplying $C(x,y)$ by the square class $\an{a}$ gives 
\begin{eqnarray*}
\an{a}C(x,y)&=&\an{\frac{(t^{-1}(x)-t(x))(x-y)}{(t^{-1}(x)-y)(x-t(x))}}
\gpb{\frac{(t^{-1}(x)-t(x))(x-y)}{(t^{-1}(x)-y)(x-t(x))}}-
\an{-1}\gpb{\frac{(t(x)-x)(t^{-1}(x)-y)}{(t(x)-y)(t^{-1}(x)-x)}}\\
&&+\an{\frac{(y-t(y))(x-t(x)}{(x-t(y))(y-t(x))}}
\gpb{\frac{(x-t(y))(t(x)-y)}{(x-y)(t(x)-t(y))}}
+\an{\frac{(y-t(y))(t(x)-x)}{(x-y)(t(x)-t(y)}}\gpb{\frac{(t(x)-t(y))(x-y)}
{(t(x)-y)(x-t(y))}}\\
\end{eqnarray*}

Now let 
\[
s=s(x,y):=\frac{(t^{-1}(x)-t(x))(x-y)}{(t^{-1}(x)-y)(x-t(x))}
\]
and
\[
u=u(x,y):=\frac{(x-y)(t(x)-t(y))}{(x-t(y))(t(x)-y)}.
\]
Thus
\begin{eqnarray*}
\an{a}C(x,y)&=&\an{s}\gpb{s}-\an{-1}\gpb{\frac{1}{1-s}}+\an{1-u}\gpb{u^{-1}}+\an{u^{-1}-1}\gpb{u}\\
&=&\an{s}\gpb{s}-\an{-1}\gpb{\frac{1}{1-s}}+\suss{2}{u}.
\end{eqnarray*}

Now 
\[
t^{-1}(x)-t(x)=\frac{x^2-x+1}{x(x-1)},\ t^{-1}(x)-y=\frac{1-y+xy}{1-x}, \ x-t(x)=-\frac{x^2-x+1}{x}
\]
so that
\[
s=\frac{x-y}{1-y+xy}.
\]

On the other hand,  
\[
u=\frac{(x-y)^2}{(xy-y+1)(xy-x+1)}=s\cdot \frac{s}{s-1}=\frac{s}{1-s^{-1}}.
\]

Hence $C(x,y)=\an{a}E$ where 
\[
E=E(x,y)=\an{s}\gpb{s}-\an{-1}\gpb{\frac{1}{1-s}}+\suss{2}{\frac{s}{1-s^{-1}}}.
\]

Now, from the definition of $\suss{2}{s}$ and the identity $\an{-1}\suss{2}{a}=\suss{2}{a^{-1}}$,  we have
\[
\an{s-1}\suss{2}{s^{-1}}=\an{1-s}\suss{2}{s}=\an{s}\gpb{s}+\gpb{s^{-1}}.
\]
Furthermore
\[
\cconst{F}=D\left(s^{-1}\right)=\gpb{s^{-1}}+\an{-1}\gpb{\frac{1}{1-s}}-\suss{1}{\frac{1}{1-s^{-1}}}.
\]
It follows that
\begin{eqnarray*}
E&=&\an{s-1}\suss{2}{s^{-1}}-\cconst{F}-\suss{1}{\frac{1}{1-s^{-1}}}+\suss{2}{\frac{s}{1-s^{-1}}}.
\end{eqnarray*}
But 
\[
\an{s-1}=\an{\frac{s}{1-s^{-1}}}.  
\]
Thus, using the cocycle property of $\suss{2}{\ }$ we get 
\begin{eqnarray*}
E&=&\an{\frac{s}{1-s^{-1}}}\suss{2}{s^{-1}}+\suss{2}{\frac{s}{1-s^{-1}}}-\cconst{F}-\suss{1}{\frac{1}{1-s^{-1}}}\\
&=& \suss{2}{\frac{1}{1-s^{-1}}}-\suss{1}{\frac{1}{1-s^{-1}}}-\cconst{F}.
\end{eqnarray*}

Now let 
\[
r:=\frac{1}{1-s^{-1}}=\frac{x-y}{x-1-xy}
\]
 and observe that 
\[
a=\frac{(x-y)(t(x)-y)}{t(x)-x}=\frac{(x-y)(x-1-xy)}{x-1-x^2}=\frac{(x-y)(x-1-xy)}{-\Phi(x)}
\]
and hence $\an{a}=\an{-\Phi(x)r}$.  

Thus 
\begin{eqnarray}\label{eqn:df}
C=C(x,y)=\an{-\Phi(x)r}(\suss{2}{r}-\suss{1}{r}-\cconst{F})\in \rbl{F}
\end{eqnarray}
has order dividing $3$ and is independent of the choice of $x$ and $y$. 

For the remainder of the proof, we will suppose that  $x$ satisfies $(x+1)\Phi(x)\not=0$. 
Since for any $r\in F^\times\setminus\{ x^{-1}, x/(x-1)\}$ we can solve for $y$ 
\[
y=\frac{rx-r-x}{rx-1}\in F^\times
\] 
 then
clearly $r$ can assume any nonzero value other than $x^{-1}$ and $x/(x-1)$ by appropriate choice of $y$. 

In particular, taking $r=-1$ and multiplying both sides of the equation by $4$ we see that
\[
C=-\an{\Phi(x)}\cconst{F}
\]
since $4\suss{i}{-1}=0$ while $4C=C$ and $4\cconst{F}=\cconst{F}$.  

It follows that $\an{\Phi(x)}\cconst{F}$ is independent of $x$. Using the identity
\[
\Phi(x_1)\Phi(x_2)=(x_1+x_2-1)^2\Phi\left(\frac{x_1x_2-1}{x_1+x_2-1}\right)
\]
it follows that 
\[
\an{\Phi(x_1)}\an{\Phi(x_2)}=\an{\Phi\left(\frac{x_1x_2-1}{x_1+x_2-1}\right)}\ (=\an{\Phi(z)},\mbox{ say}).
\]
Thus, from
\[
\an{\Phi(x_2)}\cconst{F}=\an{\Phi(z)}\cconst{F}
\]
we deduce, on multiplying by $\an{\Phi(x_1)}$, that
\[
\an{\Phi(z)}\cconst{F}=\an{\Phi(x_1)}\an{\Phi(z)}\cconst{F}
\]
and thus that $\an{\Phi(x_1)}\cconst{F}=\cconst{F}$. Since we also have $\pf{-1}\cconst{F}=0$ (by Corollary 
\ref{cor:bconst} (1)), it follows more generally that 
\[
\pf{\pm\Phi(x)z^2}\cconst{F}=0
\]
for any $z$. Note that if $E\not= F$ then a norm from $E$ will be an element of the form
$\Phi(a)b^2$ for some $a,b\in F^\times$. Thus, statement (2) of the theorem follows immediately.  

Now choose $r=b^2$ in formula (\ref{eqn:df}). This gives 
\[
C=\an{\Phi(x)}(\suss{2}{b^2}-\suss{1}{b^2})-\cconst{F}.
\] 
Multiplying both sides by $4$ again shows that $C=-\cconst{F}$.  

Using this, formula (\ref{eqn:df}) says that 
\[
-\cconst{F}=\an{-\Phi(x)r}(\suss{2}{r}-\suss{1}{r}-\cconst{F})
\]
or, equivalently,
\[
\an{-\Phi(x)r}\cconst{F}=\cconst{F}+\suss{1}{r}-\suss{2}{r}
\]
and hence
\[
\pf{-\Phi(x)r}\cconst{F}=\suss{1}{r}-\suss{2}{r}
\]
for all $r$. Since $\pf{-\Phi(x)}\cconst{F}=0$  for all $x$ and since 
\[
\pf{-\Phi(x)r}=\pf{-\Phi(x)}\pf{r}+\pf{-\Phi(x)}+\pf{r}
\]
 it follows that
\[
\pf{r}\cconst{F}=\suss{1}{r}-\suss{2}{r}
\]
for all $r\in F^\times$, proving statement (1) of the theorem.
\end{proof}

\begin{rem} We will see below that statement (2) is in general best possible.  For example if $F$ is a local 
or global field not containing a primitive cube root of unity, $\zeta_3$, 
then often we have $\pf{x}\cconst{F}\not=0$ (and hence 
$\suss{1}{x}\not=\suss{2}{x}$ ) when $x$ is not a norm from $E=F(\zeta_3)$. 
\end{rem}

\begin{rem}\label{rem:h3sl2z} 
Observe that the element $t$ chosen in the last proof lies in the image of $\spl{2}{\Z}\to\spl{2}{F}$. 

Now it is well known 
that $\spl{2}{\Z}$ can be expressed as an amalgamated product $C_4*_{C_2}C_6$. 
Here $C_4$ is cyclic of order $4$ with generator $a$ and 
$C_6$ is cyclic of order $6$ with generator $b$ with 
\[
a=\matr{0}{1}{-1}{0}\mbox{ and } b=\matr{0}{-1}{1}{-1}.
\] 
Thus $b^2=t$ is our matrix of order $3$. 
A straightforward direct calculation, using this decomposition  (see, for example, \cite[Theorem 4.1.1]{knudson:book}), 
shows that 
$\ho{3}{\spl{2}{\Z}}{\Z}$ is cyclic of order $12$ and that the inclusion $G:=\an{t}\to \spl{2}{\Z}$ induces an 
isomorphism 
$\ho{3}{G}{\Z}\cong\ptor{\ho{3}{\spl{2}{\Z}}{{\Z}}}{3}$. 

It follows that for any field $F$, the image of the natural map 
$\ptor{\ho{3}{\spl{2}{\Z}}{{\Z}}}{3}\to\ptor{\rbl{F}}{3}$ is the 
cyclic subgroup generated by $-\cconst{F}$.

\end{rem}

\section{Valuations and Specialization}\label{sec:spec}

In this section, we prove the existence of specialization or reduction maps from the refined pre-Bloch group of a 
field to the 
refined pre-Bloch group of the residue field of a valuation. We use these specialization maps to obtain lower 
bounds for the 
groups $\aug{F}\rbl{F}$ for fields with valuations.

% Let $\bconstmod{F}$ 
%be the $\sgr{F}$-submodule of $\rbl{F}$ generated by $\bconst{F}$. By Theorem \ref{thm:gacf} (1), this is group of 
%order $1$ 
%or $3$ and $\aug{F}\bconstmod{F}=0$.

\subsection{Some preliminary definitions}

In this subsection we define certain quotients of the classical and refined Bloch and pre-Bloch groups 
which we will use in the remainder of the paper.

For any field $F$ we will let 
\[
\rrpb{F}:= \rpb{F}/\ks{1}{F}.
\]

Observe that for any $x$ we have $\lambda_1(\suss{1}{x})= \pf{-x}\cdot\pf{x}\in \aug{F}^2$, by the proof of 
Lemma \ref{lem:kf}, and 
\begin{eqnarray*}
\lambda_2(\suss{1}{x})&=&(1-x)\asymm x +(1-x^{-1})\asymm x^{-1}\\
= (1-x)\asymm x -\left(\frac{1-x}{-x}\right)\asymm x &=& (1-x)\asymm x-(1-x)\asymm x+ (-x)\asymm x\\ 
&=& (-x)\asymm x \in \asym{2}{\Z}{F^\times}.
\end{eqnarray*}
It follows that $\bw(\suss{1}{x})=\rsf{-x}{x}\in\rasym{2}{\Z}{F^\times}$ so that 
$\bw(\ks{1}{F})$ is the $\sgr{F}$-submodule of $\rasym{2}{\Z}{F^\times}$ generated by the symbols $\rsf{-x}{x}$.
Thus we set
\[
\rrasym{2}{\Z}{F^\times}:= \rasym{2}{\Z}{F^\times}/\bw(\ks{1}{F})
\]
and let 
\[
\rbw:\rrpb{F}\to \rrasym{2}{\Z}{F^\times}
\]
be the resulting $\sgr{F}$-homomorphism induced by $\bw$. Finally we let $\rrbl{F}:=\ker{\rbw}$.

 \begin{lem}\label{lem:rbftilde} 
For any field $F$ the natural map $\rpb{F}\to\rrpb{F}$ induces a surjection $\rbl{F}\to\rrbl{F}$ 
whose kernel is annihilated by $4$. In particular,
\[
\zhalf{\rbl{F}}\cong\zhalf{\rrbl{F}}.
\]  
\end{lem}
\begin{proof}
The surjectivity of the map is clear from the definitions. On the other hand, the kernel of the map 
$\rbl{F}\to\rrbl{F}$ is $\rbl{F}\cap\ks{1}{F}$ which is contained in 
$\ker{\lambda_1\res{\ks{1}{F}}}$ and which in turn is annihilated by $4$ by Lemma \ref{lem:kf}. 
\end{proof}

For finite fields, the results of \cite{hut:arxivcplx11} allow us to be more precise:

\begin{lem} \label{lem:rbffin}
Let $\F{q}$ be a finite field with $q$ elements. 
\begin{enumerate}
\item If $q$ is even or if $q\equiv 1\pmod{4}$ then 
$\bl{\F{q}}=\rbl{\F{q}}=\rrbl{\F{q}}$. This group is cyclic of order $(q+1)$ when $q$ is even and $(q+1)/2$ when 
$q\equiv 1\pmod{4}$.   

\item If $q\equiv 3\pmod{4}$ then $\rrbl{\F{q}}\cong\bl{\F{q}}/\an{\sus{-1}}$ is cyclic of order $(q+1)/4$.  
\end{enumerate}
\end{lem} 

\begin{proof} The cases $q=2$ or $q=3$ are immediate.

When $q\geq 4$, by \cite[Lemma 7.1]{hut:arxivcplx11}, 
the natural map $\rpb{\F{q}}\to\pb{\F{q}}$
induces an  isomorphism $\rbl{\F{q}}\cong\bl{\F{q}}$ and for all 
$x$ the image of $\suss{1}{x}\in\rbl{\F{q}}$ is $\sus{x}\in \bl{\F{q}}$, which has order divisible by $2$. 

If $q$ is even or if $q\equiv 1\pmod{4}$ then 
$\bl{\F{q}}$ is cyclic of order $q+1$ or $(q+1)/2$ respectively, and hence is  of odd order. 

On the other hand, if $q\equiv 3\pmod{4}$ then \cite[Lemma 7.8]{hut:arxivcplx11} shows that $\ks{1}{\F{q}}=
\an{\suss{1}{-1}}$ is cyclic of order $2$ and is contained in $\rbl{\F{q}}\cong\bl{\F{q}}$.
\end{proof}

We let $\kks{F}$ denote the $\sgr{F}$-submodule $\ks{1}{F}+\ks{2}{F}$ of $\rpb{F}$.  

\begin{lem}\label{lem:kks}
For any field $F$ we have 
\[
\kks{F}=\ks{1}{F}+\aug{F}\bconst{F}
\]
and 
\[
\kks{F}\cong \ks{1}{F}\oplus\aug{F}\cconst{F}.
\]
\end{lem}
\begin{proof}
By Theorem \ref{thm:df}, for any field $F$  we have an identity of  $\sgr{F}$-submodules of $\rpb{F}$ 
\[
\kks{F}=\ks{1}{F}+\ks{2}{F}=\ks{1}{F}+\aug{F}\cconst{F}.
\]
On the other hand since 
\begin{eqnarray*}
\cconst{F}\equiv\gpb{\frac{1}{x}}+\an{-1}\gpb{\frac{1}{1-x}}\pmod{\ks{1}{F}}\\
\bconst{F}\equiv\gpb{x}+\an{-1}\gpb{1-x}\pmod{\ks{1}{F}}\\
\mbox{ and } \gpb{x^{-1}}\equiv -\an{-1}\gpb{x}\pmod{\ks{1}{F}},
\end{eqnarray*}
it follows that $\bconst{F}\equiv-\an{-1}\cconst{F}\equiv-\cconst{F}\pmod{\ks{1}{F}}$ 
(using Corollary \ref{cor:bconst} (1)) and thus 
\[
\ks{1}{F}+\aug{F}\bconst{F}=\ks{1}{F}+\aug{F}\cconst{F}.
\]

Finally, suppose that $a\in \ks{1}{F}\cap\aug{F}\cconst{F}$. Then $3a=0$ since $3\cconst{F}=0$ and $4a=0$ 
since $a\in \ks{1}{F}\cap\rbl{F}$, so that $a=0$. Thus $\ks{1}{F}\cap\aug{F}\cconst{F}=0$.
\end{proof}

For any field $F$ we set  
\[
\rrrpb{F}:=\rpb{F}/\kks{F}=\rrpb{F}/\aug{F}\cconst{F}\mbox{ and } \rrrbl{F}=\rrbl{F}/\aug{F}\cconst{F}.
\]

\begin{cor}\label{cor:rrrbl}
For any field $F$ there is natural short exact sequence
\[
0\to\aug{F}\cconst{F}\to\rrbl{F}\to\rrrbl{F}\to 0.
\]
\end{cor}

If $F$ is a finite field, by Lemma \ref{lem:blochfin}, $\rbl{F}\cong\bl{F}$ and hence 
the action of $\sq{F}$ on $\rbl{F}$ is trivial and thus we have 
\begin{cor} For any finite field $F$, we have  $\rrrbl{F}=\rrbl{F}$. 
\end{cor}

We will use the following identities repeatedly below:
\begin{lem}\label{lem:rrrpb}
For any field $F$ and any $x\in F^\times$, we have 
\[
\an{x}\gpb{x}=\an{-1}\gpb{x}=-\gpb{x^{-1}}
\]
 in $\rrrpb{F}$.
\end{lem}
\begin{proof}
Since $\suss{1}{x}=0$ in $\rrrpb{F}$, it follows that $\an{-1}\gpb{x}=-\gpb{x^{-1}}$. 
Since $\suss{2}{x}=0$, we have $\an{x}\gpb{x}=-\gpb{x^{-1}}$ also.
\end{proof}

We will also need to consider the corresponding quotient of the classical pre-Bloch group:

Let $\kl{F}$ denote the ($2$-torsion) subgroup of $\pb{F}$ generated by the elements $\sus{x}$, and let
\[
\redpb{F}:=\frac{\pb{F}}{\kl{F}}
\]
Thus, the natural map $\rpb{F}\to\pb{F}$ induces an ismorphism $\rrrpb{F}_{F^\times}\cong \redpb{F}$.

Recall that $\cconstmod{F}$ denotes the $\sgr{F}$-module $\sgr{F}\cconst{F}$. Finally, we let 
\[
\rrrrpb{F}:=\frac{\rpb{F}}{\ks{1}{F}+\cconstmod{F}}=\frac{\rrpb{F}}{\cconstmod{F}}=
\frac{\rrrpb{F}}{\Z\cdot\cconst{F}}
\]  
and
\[
\rrrrbl{F}:=\frac{\rrbl{F}}{\cconstmod{F}}.
\]

In section \ref{sec:local} below, we will also need the corresponding classical version:
\[
\rredpb{k}:=\frac{\pb{k}}{\Z\cconst{k}+\kl{k}}=\frac{\redpb{k}}{\Z\cconst{k}}. 
\]

Note that $\cconstmod{F}=\aug{F}\cconst{F}+\Z\cconst{F}$ and that the sum
$\ks{1}{F}+\cconstmod{F}$ is direct (by the argument of Lemma \ref{lem:kks}). Thus we have 
\begin{lem} 
\label{lem:rrrrbl}
For any field $F$ there are short exact sequences
\[
0\to\cconstmod{F}\to\rrbl{F}\to\rrrrbl{F}\to 0
\]
and
\[
0\to \Z\cconst{F}\to \rrrbl{F}\to \rrrrbl{F}\to 0.
\]
\end{lem}

\subsection{Valuations}

Now let $\Gamma$ be an ordered (additive) abelian group and let $v:F^\times\to \Gamma$ be a valuation, 
\emph{which we will always assume to be 
surjective}. As usual, let 
\[
\mathcal{O}=\mathcal{O}_v=\zero\cup\{ x\in F^\times\ |\ v(x)\geq 0\}
\]
 be the valuation ring, let  
\[
\mathcal{M}=\mathcal{M}_v=\zero\cup \{x\in F^\times\ |\ v(x)>0\}
\]
be  the maximal ideal, let $k=k_v=\mathcal{O}/\mathcal{M}$ be the residue 
field and 
\[
U=U_v=\{ x\in F^\times\ |\ v(x)=0\}=\mathcal{O}\setminus\mathcal{M}.
\]

 Also let $f_v$ be the group homomorphism $U\to k^\times$, $u\mapsto\bar{u}:=u+\mathcal{M}$, and let  
$U_1=U_{1,v}=\ker{f_v}=1+\mathcal{M}$.

%For an additive abelian group $A$, we let $\modtwo{A}$ denote $A\otimes\Z/2$, while for multiplicative groups $M$ 
%we will continue to use the notation $M/M^2$.

Since $\Gamma$ is a torsion-free group we have a short exact sequence of $\F{2}$-vector spaces 
\begin{eqnarray}\label{eqn:split}
1\to U/U^2\to\sq{F}\to\modtwo{\Gamma}\to 0.
\end{eqnarray}

We have homomorphisms of commutative rings
\begin{eqnarray*}
\xymatrix{
\Z[U/U^2]\ar@{^(->}[r]\ar@{>>}[d]
&\sgr{F}\\
\sgr{k}
&{}
}
\end{eqnarray*}

Thus, if $M$ is any $\sgr{k}$ module 
\[
M_F=M_{F,v}:=\sgr{F}\otimes_{\Z[U/U^2]}M
\]
is an $\sgr{F}$-module.

We will require the following result several times in the next section:

\begin{lem}\label{lem:cancel}
Let $F$ be a field with  valuation $v:F^\times\to\Gamma$ 
and corresponding residue field $k$. Let $a,b\in F$ and suppose that 
$v(a)\equiv v(b)\pmod{2\Gamma}$. Then $\an{a}\otimes\bconst{k}=\an{b}\otimes\bconst{k}$ in $\rrrpb{k}_F$.
\end{lem}

\begin{proof}
We have $v(a)=v(b)+2\gamma$ for some $\gamma\in \Gamma$. Choose $c\in F^\times$ with $v(c)=\gamma$. Then
$u=a/(bc^2)\in U$. So $\an{a}-\an{b}=\an{b}\pf{u}$ in $\sgr{F}$. Thus
\[
\an{a}\otimes\bconst{k}-\an{b}\otimes\bconst{k}=\an{a}\otimes\pf{\bar{u}}\bconst{k}=0
\]
since $\pf{\bar{u}}\bconst{k}\in \aug{k}\bconst{k}$ and hence represents $0$ in $\rrrpb{k}$.
\end{proof}
\subsection{The specialization homomorphisms}

\begin{thm}\label{thm:spec}
Let $F$ be a field with  valuation $v$ and corresponding residue field $k$. 

Then there is a surjective $\sgr{F}$-module homomorphism   
\begin{eqnarray*}
S_{v}:\rpb{F}&\to&\rrrpb{k}_F\\
\gpb{a}&\mapsto&\left\{
\begin{array}{rc}
1\otimes \gpb{\bar{a}},& v(a)=0\\
1\otimes\bconst{k},& v(a)>0\\
-(1\otimes\bconst{k}),& v(a)<0\\
\end{array}
\right.
\end{eqnarray*}
\end{thm}
\begin{proof}
Let $Z_1$ denote the set of symbols of the form $\gpb{x}, x\not= 1$ and let 
$T:\sgr{F}[Z_1]\to \rrrpb{k}_F$ be the unique $\sgr{F}$-homomorphism given by  
\begin{eqnarray*}
\gpb{a}&\mapsto&\left\{
\begin{array}{rc}
1\otimes \gpb{\bar{a}},& v(a)=0\\
1\otimes\bconst{k},& v(a)>0\\
-(1\otimes\bconst{k}),& v(a)<0\\
\end{array}
\right.
\end{eqnarray*}
We must prove that $T(S_{x,y})=0$ for all $x,y\in F^\times\setminus\{ 1\}$.

Through the remainder of this proof we will adopt the following notation: Given $x,y\in F^\times\setminus\{ 1\}$, 
we let
\[
u=\frac{y}{x}\mbox{ and } w=\frac{1-x}{1-y}.
\]
Note that 
\[
\frac{1-x^{-1}}{1-y^{-1}}=\frac{y}{x}\cdot\frac{x-1}{y-1}=uw.
\]
Thus, with this notation, $S_{x,y}$ becomes $\gpb{x}-\gpb{y}+\an{x}\gpb{u}-\an{x^{-1}-1}\gpb{uw}+\an{1-x}\gpb{w}$.

We divide the proof into several cases:
\begin{enumerate}
\item[Case (i):] $v(x),v(y)\not= 0$
\begin{enumerate}
\item[Subcase (a):] $v(x)=v(y)>0$. 

Then $1-x,1-y\in U_1$ and hence $w\in U_1$, so that $\bar{w}=1$ and $\overline{uw}=\bar{u}$. Thus 
\[
T(S_{x,y})=1\otimes\bconst{k}-1\otimes\bconst{k}+\an{x}\otimes\gpb{\bar{u}}-\an{x^{-1}-1}\otimes\gpb{\bar{u}}.
\]
However, $x^{-1}-1=x^{-1}(1-x)$, so that 
$\an{x^{-1}-1}\otimes\bar{u}=\an{x}\otimes\an{1-\bar{x}}\gpb{\bar{u}}=\an{x}\otimes\gpb{\bar{u}}$, and thus 
$T(S_{x,y})=0$ as required.

\item[Subcase (b):] $v(x)=v(y)<0$.

Then $u\in U$ and $uw\in U_1$ so that $\bar{w}=\bar{u}^{-1}$. Thus
\[
T(S_{x,y})=-1\otimes\bconst{k}+1\otimes\bconst{k}+\an{x}\otimes\gpb{\bar{u}}+\an{1-x}\otimes\gpb{\bar{u}^{-1}}.
\]

But $1-x=-x(1-x^{-1})$ and $1-x^{-1}\in U_1$, so that the last term is $\an{-x}\otimes\gpb{\bar{u}^{-1}}$ and 
hence $T(S_{x,y})=\an{x}\otimes\suss{1}{\bar{u}}=0$ in $\rrrpb{k}_F$.

%\item[Case (ii):] $v(x),v(y)\not= 0$ and $v(x)\not= v(y)$.

\item[Subcase (c):] $v(x)>v(y)>0$.

Then $w\in U_1$ and $v(u), v(uw)<0$. So
\[
T(S_{x,y})=1\otimes \bconst{k}-1\otimes\bconst{k}-\an{x}\otimes\bconst{k}+\an{x^{-1}-1}\otimes\bconst{k}.
\]
But since $x^{-1}-1=x^{-1}(1-x)$ and $1-x\in U_1$ it follows that $\an{x^{-1}-1}\otimes\bconst{k}=\an{x}\otimes\bconst{k}$
and hence $T(S_{x,y})=0$.
\item[Subcase (d):] $v(x)>0>v(y)$.

Then $v(u)=v(y)-v(x)<0$, $v(w)=-v(1-y)=-v(y(y^{-1}-1))=-v(y)>0$ and $v(uw)=v(u)+v(w)=-v(x)<0$. So 
\[
T(S_{x,y})=1\otimes\bconst{k}+1\otimes\bconst{k}-\an{x}\otimes \bconst{k}+\an{x^{-1}-1}\otimes\bconst{k}+
\an{1-x}\otimes\bconst{k}.
\]
But $1-x\in U_1$ and $\an{x^{-1}-1}=\an{x}\an{1-x}$. So this gives $T(S_{x,y})=1\otimes 3\bconst{k}$. However,
$3\bconst{k}=-3\cconst{k}=0$ in $\rrrpb{k}$.
\item[Subcase (e):] $0>v(x)>v(y)$.

Then $v(u)=v(y)-v(x)<0$ and $uw\in U_1$, so that $v(w)=-v(u)>0$ and $\bar{w}=\bar{u}^{-1}$. Thus 
\[
T(S_{x,y})=-(1\otimes\bconst{k})+1\otimes\bconst{k}-\an{x}\otimes\bconst{k}+\an{1-x}\otimes\bconst{k}.
\]
Now $1-x=-x(1-x^{-1})$ and $1-x^{-1}\in U_1$, so the last term is 
$\an{x}\otimes\an{-1}\bconst{k}=\an{x}\otimes\bconst{k}$ by Corollary \ref{cor:bconst} (1). This gives 
$T(S_{x,y})=0$ as required.
\item[Subcases (f),(g),(h):] The corresponding calculations when $v(y)>v(x)$ are almost identical.
\end{enumerate}
\item[Case (ii):] $x,y\in U_1$.
\begin{enumerate}
\item[Subcase (a):] $v(1-x)\not= v(1-y)$.

Then $u\in U_1$ and $v(w)=v(uw)\not= 0$. So 
\[
T(S_{x,y})=\pm\left( \an{x^{-1}-1}\otimes \bconst{k}-\an{1-x}\otimes\bconst{k}\right)=0
\]
since $\an{x^{-1}-1}=\an{x^{-1}}\an{1-x}$ and $x^{-1}\in U_1$.

\item[Subcase (b):] $v(1-x)=v(1-y)$.

Then $u\in U_1$ and $w,uw\in U$ with $\bar{uw}=\bar{w}$. So
\[
T(S_{x,y})=-\an{x^{-1}-1}\otimes\gpb{\bar{w}}+\an{1-x}\otimes\gpb{\bar{w}}
\]
which is $0$ by the same argument as the previous (sub)case.
\end{enumerate}
\item[Case (iii):] $x\in U_1$, $v(y)\not=0$. 

Then $v(u)=v(y)$. Observe that $v(1-y)=\mathrm{min}(v(1),v(y))=\mathrm{min}(0,v(y))\leq 0$. Of course, $v(1-x)>0$.

Thus $v(w)=v(1-x)-v(1-y)>0$ and $v(uw)=v(u)+v(w)=v(1-x)+v(y)-v(1-y)>0$ 
since $v(y)-v(1-y)\geq 0$. So 
\[
T(S_{x,y})=\pm(1\otimes\bconst{k}-1\otimes\bconst{k})-\an{x^{-1}-1}\otimes \bconst{k}+\an{1-x}\otimes\bconst{k}=0
\]
since $x\in U_1$ and thus $\an{x^{-1}-1}=\an{1-x}$.

\item[Case (iv):] $v(x)\not=0$, $y\in U_1$

Arguing as in the last case, $v(w),v(uw)<0$ in this case.

\begin{enumerate}
\item[Subcase (a):] $v(x)>0$

Then $v(u)=-v(x)<0$. So 
\[
T(S_{x,y})=1\otimes \bconst{k}-\an{x}\otimes\bconst{k}+\an{x^{-1}-1}\otimes\bconst{k}-\an{1-x}\otimes\bconst{k}=0
\]
using Lemma \ref{lem:cancel} together with the identities $v(1-x)=0=v(1)$ and $v(x^{-1}-1)=-v(x)$.

\item[Subcase (b):] $v(x)<0$.

Then $v(u)>0$ and
\[
T(S_{x,y})=-1\otimes \bconst{k}+\an{x}\otimes\bconst{k}+\an{x^{-1}-1}\otimes\bconst{k}-\an{1-x}\otimes\bconst{k}.
\]
This vanishes by Lemma \ref{lem:cancel} together with the identities $v(1-x)=v(x)$ and $v(x^{-1}-1)=0=v(1)$.

\end{enumerate}

\item[Case (v):] $x\in U\setminus U_1$ and $v(y)\not=0$.

\begin{enumerate}
\item[Subcase (a):] $v(y)>0$.

Then $v(u)=v(y)>0$. Since $1-x\in U$ and  $1-y\in U_1$, it follows that $v(w)=0$, $\bar{w}=1-\bar{x}$ in $k$ and 
$v(uw)>0$. Thus 
\begin{eqnarray*}
T(S_{x,y})&=&1\otimes\gpb{\bar{x}}-1\otimes\bconst{k}+\an{x}\otimes\bconst{k}-\an{x^{-1}-1}\otimes\bconst{k}
+\an{1-x}\otimes\gpb{1-\bar{x}}\\
&=&1\otimes(\gpb{\bar{x}}+\an{1-\bar{x}}\gpb{1-\bar{x}}-\bconst{k})
\end{eqnarray*}
using Lemma \ref{lem:cancel} to eliminate terms.
%Now $\pf{1-\bar{x}}\bconst{k}\in \aug{k}\bconst{k}$, so this term is $0$ in $\rrrpb{k}$. 
Now 
$\an{1-\bar{x}}\gpb{1-\bar{x}}=\an{-1}\gpb{1-\bar{x}}$ in $\rrrpb{k}$ by Lemma \ref{lem:rrrpb}, and hence 
$\gpb{\bar{x}}+\an{1-\bar{x}}\gpb{1-\bar{x}}=\bconst{k}$ in $\rrrpb{k}$. Thus we conclude that 
$T(S_{x,y})=0$ as required. 
\item[Subcase (b):] $v(y)<0$

We have $v(u)=v(y)<0$ and $v(w)=-v(1-y)=-v(y)>0$. Thus $v(uw)=v(u)+v(w)=v(y)-v(y)=0$. Furthermore, since 
$1-y^{-1}\in U_1$, $\bar{uw}=1-\bar{x}^{-1}$. 

Thus (using Lemma \ref{lem:cancel} again)
\begin{eqnarray*}
S(T_{x,y})& = & 1\otimes \gpb{\bar{x}}+1\otimes \gpb{k}-\an{x}\otimes\bconst{k}-\an{x^{-1}-1}\otimes\gpb{1-\bar{x}^{-1}}
+\an{1-x}\otimes \bconst{k}\\
&=&1\otimes\left( \bconst{k}+\gpb{\bar{x}}-\an{\bar{x}^{-1}-1}\gpb{1-\bar{x}^{-1}}\right)\\
&=& 1\otimes\left( \bconst{k}+\gpb{\bar{x}}-\an{\bar{x}^{-1}-1}\gpb{1-\bar{x}^{-1}}\right).
\end{eqnarray*}
Now $\an{\bar{x}^{-1}-1}\gpb{1-\bar{x}^{-1}}=\gpb{1-\bar{x}^{-1}}$ in $\rrrpb{k}$ by Lemma \ref{lem:rrrpb}, and 
$\gpb{\bar{x}}=-\an{-1}\gpb{\bar{x}^{-1}}$ since $\suss{1}{\bar{x}}=0$. Thus 
\[
\gpb{\bar{x}}-\an{\bar{x}^{-1}-1}\gpb{1-\bar{x}^{-1}}=-\an{-1}(\gpb{\bar{x}^{-1}}+\an{-1}\gpb{1-\bar{x}^{-1}})=
-\an{-1}\bconst{k}=-\bconst{k}\mbox{ in } \rrrpb{k}
\]
and again $T(S_{x,y})=0$ as required.
\end{enumerate}
\item[Case (vi):] $y\in U\setminus U_1$ and $v(x)\not= 0$
\begin{enumerate}
\item[Subcase (a):] $v(x)>0$

we have $v(u)=-v(x)<0$, $v(w)=0$ and $1-x\in U_1$ so that $\bar{w}=(1-\bar{y})^{-1}$. Finally $v(uw)=v(u)<0$. Thus
\begin{eqnarray*}
T(S_{x,y})&=&1\otimes\bconst{k}-1\otimes\gpb{\bar{y}}-\an{x}\otimes\bconst{k}+\an{x^{-1}-1}\otimes\bconst{k}
+\an{1-x}\otimes\gpb{\frac{1}{1-\bar{y}}}\\
&=& 1\otimes\left( \bconst{k}-\gpb{\bar{y}}+\gpb{\frac{1}{1-\bar{y}}}\right)
\end{eqnarray*}
using Lemma \ref{lem:cancel} again (after observing that $v(x^{-1}-1)=-v(x)$ when $v(x)>0$). 

However 
\[
\gpb{\frac{1}{1-\bar{y}}}=-\an{-1}\gpb{1-\bar{y}}
\]
in $\rrrpb{k}$, and hence $T(S_{x,y})=1\otimes(\bconst{k}-\bconst{k})=0$.
\item[Subcase (b):] $v(x)<0$

We have $v(u)=-v(x)>0$ and $v(w)=v(1-x)=v(x)<0$ and $v(uw)=v(u)+v(w)=0$. 
Furthermore, since $1-x^{-1}\in U_1$, $\bar{uw}=1/(1-\bar{y}^{-1})$. Thus
\begin{eqnarray*}
T(S_{x,y})&=& -(1\otimes \bconst{k})-1\otimes\gpb{\bar{y}}+\an{x}\otimes\bconst{k}
-\an{x^{-1}-1}\otimes\gpb{\frac{1}{1-\bar{y}^{-1}}}-\an{1-x}\otimes\bconst{k}\\
&=& 1\otimes \left(-\bconst{k}-\gpb{\bar{y}}-\an{-1}\gpb{\frac{1}{1-\bar{y}^{-1}}}\right) 
\end{eqnarray*}
using Lemma \ref{lem:cancel} and the fact that $\an{x^{-1}-1}=\an{-1}\an{1-x^{-1}}$ and $1-x^{-1}\in U_1$.

But in $\rrrpb{k}$ we have, by Lemma \ref{lem:rrrpb}, 
\[
-\gpb{\bar{y}}-\an{-1}\gpb{\frac{1}{1-\bar{y}^{-1}}}=\an{-1}\gpb{\bar{y}^{-1}}+\gpb{1-\bar{y}^{-1}}=\bconst{k}.
\]

\end{enumerate}
\item[Case (vii):] $x\in U\setminus U_1$ and $y\in U_1$.

We have $v(u)=0$ and $\bar{u}=\bar{x}^{-1}$. Furthermore, $v(w)=-v(1-y)<0$ and thus $v(uw)=v(w)<0$ also.
Thus
\begin{eqnarray*}
T(S_{x,y})&=&1\otimes\gpb{\bar{x}}+\an{x}\otimes\gpb{\bar{x}^{-1}}+\an{x^{-1}-1}\otimes\bconst{k}
-\an{1-x}\otimes\bconst{k}\\
&=&1\otimes \left( \gpb{\bar{x}}+\an{\bar{x}}\gpb{\bar{x}^{-1}}\right)\mbox{ (by Lemma \ref{lem:cancel})}\\
&=&1\otimes\an{\bar{x}}\an{1-\bar{x}}\suss{2}{\bar{x}}=0.
\end{eqnarray*}
\item[Case (viii):] $x\in U_1$ and $y\in U\setminus U_1$

We have $u\in U$ and $\bar{u}=\bar{y}$, $v(w)=v(1-x)>0$ and $v(uw)=v(w)>0$. So
\begin{eqnarray*}
T(S_{x,y})&=& -(1\otimes\gpb{\bar{y}})+\an{x}\otimes\gpb{\bar{y}}-\an{x^{-1}-1}\otimes\bconst{k}+
\an{1-x}\otimes\bconst{k}=0
\end{eqnarray*}
by Lemma \ref{lem:cancel} and the fact that $x\in U_1$.
\item[Case (ix):] $x,y\in U\setminus U_1$

In this case $\bar{x},\bar{y}\in k^\times\setminus \{ 1\}$ and 
\[
T(S_{x,y})=1\otimes S_{\bar{x},\bar{y}}=1\otimes 0.
\]
\end{enumerate}
\end{proof}

Given a valuation $v$ on a field $F$ and an element $a\in F^\times$ with $v(a)\not= 0$, we will let 
$\epsilon_v(a)$ denote $\mathrm{sign}(v(a))\in\{ \pm 1\}$.  

Suppose that $\rho:\sgr{F}\to\sgr{k}$ is a  ring homomorphism satisfying $\rho(\an{u})=\an{\bar{u}}$ for any $u\in U$.
Then by composing $S_v$ with the surjective $\sgr{F}$-module homomorphism 
\[
\rrrpb{k}_F\to\rrrpb{k},\ x\otimes y\mapsto \rho(x)y.
\]
we obtain:

\begin{cor} \label{cor:srho}
Let $\rho:\sgr{F}\to\sgr{k}$ be any ring homomorphism satisfying $\rho(\an{u})=\an{\bar{u}}$ if $u\in U$. Then $\rho$
induces a $\sgr{F}$-module structure on $\rpb{k}$ and there is a surjective homomorphism of $\sgr{F}$-modules 
$S=S_\rho:\rpb{F}\to \rrrpb{k}$ determined by 
\begin{eqnarray*}
\gpb{a}&\mapsto&\left\{
\begin{array}{rc}
 \gpb{\bar{a}},& a\in U_v\\
\epsilon_v(a)\bconst{k},& a\not\in U_v\\
\end{array}
\right.
\end{eqnarray*}
\end{cor}

Now, if we choose a splitting, $j:\modtwo{\Gamma}\to\sq{F}$ of the sequence (\ref{eqn:split}), 
we obtain a ring isomorphism
\[
\sgr{F}\cong\Z[U/U^2\times j(\modtwo{\Gamma})]\cong \Z[U/U^2]\otimes_\Z\Z[\modtwo{\Gamma}].
\] 

On the other hand, any character $\chi:\modtwo{\Gamma}\to \mu_2=\{ 1,-1\}$ induces a ring homomorphism 
$\Z[\modtwo{\Gamma}]\to \Z$. Thus, given a splitting $j$ and a character $\chi$, we obtain a ring homomorphism
\[
\rho=\rho_{j,\chi}:\sgr{F}\to\sgr{k}, \an{u\cdot j(\gamma)}\mapsto \chi(\gamma)\an{\bar{u}}
\]
and a corresponding surjective specialization map 
\[
S=S_{j,\chi}:\rpb{F}\to \rrrpb{k}
\]
which is also an $\sgr{F}$-homomorphism. 

\begin{exa}\label{exa:dv}
For example, if $v$ is a discrete valuation, then $\Gamma=\Z$ and $\modtwo{\Gamma}=\Z/2$. 
Any choice, $\pi$, of uniformizing 
parameter, determines a splitting $\Z/2\to\sq{F},\  1\mapsto \an{\pi}$. There are two characters, $\epsilon$, on $\Z/2$, 
namely $1$ and $-1$. Thus we get ring homomorphisms 
\[
\rho_{\pi,\epsilon}:\sgr{F}\to \sgr{k},\an{u\pi^r}\mapsto \epsilon^r\an{\bar{u}}. 
\]
 and corresponding specialization homomorphisms 
\[
S_{\pi,\epsilon}:\rpb{F}\to\rrrpb{k}.
\]
\end{exa}

In general, the specialization homomorphisms $S_{j,\chi}$ do not restrict to homomorphisms 
$\rbl{F}\to\rrrbl{k}$ of Bloch groups. 
%the images of these maps, when restricted to $\rbl{F}$, are often strictly larger than $\rrrbl{k}$.

To see this, we begin with the following observation:
\begin{lem}\label{lem:ap}
For a field $F$ and $x\in F^\times$ we have
\[
\pp{-1}\pf{x}\rrpb{F}\subset \rrbl{F}.
\]
\end{lem}
\begin{proof}
Since $\sq{F}$ acts trivially on $\asym{2}{\Z}{F^\times}$ we have $\pf{x}\asym{2}{\Z}{F^\times}=0$ and thus
\begin{eqnarray*}
\rbw(\pf{x}\pp{-1}\gpb{a})&=&(\pf{x}\pf{1-a}\pp{-1}\pf{a},0)\\
&=&(\pf{x}\pf{1-a}\pf{-a}\pf{a},0)\\
&=&\pf{x}\pf{1-a}\rsf{-a}{a}=0\in \rrasym{2}{\Z}{F^\times}.
\end{eqnarray*}
\end{proof}

\begin{cor}\label{cor:spec}
Let $F$ be a field with valuation $v:F^\times\to \Gamma$, with corresponding residue field $k$. Let 
$j:\modtwo{\Gamma}\to \sq{F}$ be a splitting, and let $\chi$ be a \emph{nontrivial} character on $\modtwo{\Gamma}$. 

Then the image of 
\[
S_{j,\chi}\res{\rbl{F}}:\rbl{F}\to\rrrpb{k}
\] 
contains $2\pp{-1}\rrrpb{k}$. 
\end{cor}
\begin{proof} First observe that $S_{j,\chi}$ induces a homomorphism 
\[
\rrpb{F}\to\rrrpb{k}
\]
since 
\[
S_{j,\chi}(\suss{1}{x})=
\left\{
\begin{array}{ll}
\suss{1}{\bar{x}},& v(x)=0\\
0,& v(x)\not= 0
\end{array}
\right.
\]
Suppose that $\gamma\in \Gamma$ with $\chi(\gamma)=-1$ and let $\pi:=j(\gamma)$. Since $\rrrpb{k}$ is an $\sgr{F}$-module
via the homomorphism $\rho_{j,\chi}$, $\an{\pi}$ acts on $\rrrpb{k}$ as multiplication by $-1$. Thus 
$\pf{\pi}\rrrpb{k}=2\rrrpb{k}$ and hence $S_{j,\chi}$ induces a surjective homomorphism
\[
\pf{\pi}\pp{-1}\rrpb{F}\to2\pp{-1}\rrrpb{k}.
\] 
\end{proof}
\begin{rem} As just noted, the maps $S_{j,\chi}:\rpb{F}\to\rrrpb{k}$ descend to maps $\rrpb{F}\to\rrrpb{k}$. However, 
they do not usually induce maps $\rrrpb{F}\to\rrrpb{k}$.
In fact, below we will use the homomorphisms $S_{j,\chi}$ to detect the non-triviality of $\aug{F}\bconst{F}$ for 
global and (some) local fields.
\end{rem}
\begin{cor}
Let $F$ be a field with valuation $v:F^\times\to\Gamma$ and \emph{quadratically closed} residue field $k$. 
 Let $j:\modtwo{\Gamma}\to\sq{F}$ be 
a splitting, and let $\chi:\modtwo{\Gamma}\to\mu_2$ be a \emph{nontrivial} character. Then $S_{j,\chi}$ induces 
a surjective map
\[
\xymatrix{
\rbl{F}\ar@{>>}[r]
&\rpb{k}.}
\] 
\end{cor}

\begin{proof} Since $k$ is quadratically closed, $\rpb{k}=\pb{k}$ and $\suss{i}{x}=\sus{x}=0$ for all $x$ for $i=1,2$. 
Thus $\rrrpb{k}=\pb{k}$ and $\pp{-1}\rrrpb{k}=\pp{-1}\pb{k}=2\pb{k}$. 
However, since every element of $k^\times$ is a square, $\pb{k}$ is $2$-divisible (\cite[Lemma 5.8]{sah:dupont})and thus 
$2\pp{-1}\rpb{k}=4\pb{k}=\pb{k}=\rpb{k}$.
\end{proof}

\begin{cor}\label{cor:epmrbl}
Let $F$ be a field with valuation $v:F^\times\to\Gamma$ and corresponding residue field $k$. 
Let $j:\modtwo{\Gamma}\to\sq{F}$ be 
a splitting, and let $\chi:\modtwo{\Gamma}\to\mu_2$ be a \emph{nontrivial} character. Choose $\gamma\in \Gamma$ with 
$\chi(\gamma)=-1$ and let $\pi=j(\gamma)$.

Then the map $S_{j,\chi}$ induces a surjective homomorphism of $\zhalf{\sgr{F}}$-modules
\[
\epm{\pi}\ep{-1}\left( \zhalf{\rbl{F}}\right)\to \ep{-1}\left( \zhalf{\rrrpb{k}}\right)
\]

In particular, if $-1\in (F^\times)^2$, then there is a surjective homomorphism 
\[
\epm{\pi}\left( \zhalf{\rbl{F}}\right)\to \zhalf{\rrrpb{k}}
\]
\end{cor}
\begin{proof}
Since 
\[
\epm{\pi}\ep{-1}\left( \zhalf{\rrpb{F}}\right)\subset \zhalf{\rrbl{F}}
\]
by Lemma \ref{lem:ap}, it follows that
\[
\epm{\pi}\ep{-1}\left( \zhalf{\rrpb{F}}\right)=\epm{\pi}\ep{-1}\left( \zhalf{\rrbl{F}}\right)\cong
\epm{\pi}\ep{-1}\left( \zhalf{\rbl{F}}\right)
\]
(using Lemma \ref{lem:rbftilde})
\end{proof}

\begin{rem}\label{rem:chi} 
In general, the specialization maps $S_{j,\chi}$ depend on the choice of splitting $j$ and character $\chi$.
 
To spell this out: If $\gamma\in \Gamma$ then the $\sgr{F}$-module structure on $\rrrpb{k}$ associated 
to $S_{j,\chi}$ requires that $\an{j(\gamma)}$ acts as multiplication by $\chi(\gamma)$. If $j'$ is another splitting, 
then we will have $j(\gamma)=j'(\gamma)\cdot u$ for some $u\in U$, and the action of $\an{j(\gamma)}$ with respect to the 
module-structure associated to the pair $(j',\gamma)$ will be multipication by $\an{\bar{u}}\chi(\gamma)$. Note that
$\an{\bar{u}}\in \sgr{k}$.

We can free ourselves of the dependence on the choice of splitting $j$ 
(or the choice of uniformizer $\pi$, in the case of 
a discrete valuation) by composing $S_{j,\chi}$ with the natural surjective map 
\[
\rrrpb{k}\to\redpb{k}.
\]

%(which is $\pb{k}$ modulo the subgroup 
%generated by the elements $\gpb{x}+\gpb{x^{-1}}$). 

Since $\sgr{k}$ acts trivially on $\pb{k}$, we get well-defined surjective specialization maps 
\[
S_\chi:\rpb{F}\to \redpb{k}
\]
which depend only on the choice of character $\chi:\modtwo{\Gamma}\to \mu_2$. 

It is important to note that although $\redpb{k}$ is a trivial $\sgr{k}$-module, the $\sgr{F}$-module structure induced 
from a character $\chi$ is not usually trivial.
 
Thus, if $\chi$ is not the trivial character and if $\pi\in\sq{F}$ satisfies $\chi(v(\pi))=-1$, then $\pi$ acts 
as multiplication by $-1$ on $\redpb{k}$ and $S_\chi$ 
induces a surjective homomorphism 
\[
\epm{\pi}\left( \zhalf{\rbl{F}}\right)\to \zhalf{\redpb{k}}
\]
\end{rem}

\section{Applications: global fields}
From the existence of  these specialization maps it follows that if $F$ is field with many valuations, then 
$\ho{3}{\spl{2}{F}}{{\Z}}_0$ must be large.  For example:
%We use these specialization maps to prove:
\begin{thm}\label{thm:global}
Let $F$ be a field and let $\mathcal{V}$ be a family of discrete values on $F$ satisfying 
\begin{enumerate}
\item For any $x\in F^\times$, $v(x)=0$ for all but finitely many $v\in\mathcal{V}$
\item  The map 
\[
F^\times\to\oplus_{v\in\mathcal{V}}\modtwo{\Z},\quad a\mapsto \{ v(a)\}_v
\]
is surjective.
%Given any finite subset $\{ v_1,\ldots,v_t\}$ of $\mathcal{V}$, there exist $z_1,\ldots,z_t\in F^\times$ satisfying 
%$v_i(z_j)=\delta_{i,j}$.
\end{enumerate}
 Then there is a natural  surjective homomorphism 
%\[
%\ho{3}{\spl{2}{F}}{\zhalf{\Z}}\to \left(\zhalf{\kind{F}}\right)
%\oplus\left(\bigoplus_{v\in\mathcal{V}}\zhalf{\redpb{k_v}}\right).
%\]
%inducing a surjection
\[
\ho{3}{\spl{2}{F}}{\zhalf{\Z}}_0\to \bigoplus_{v\in\mathcal{V}}\zhalf{\redpb{k_v}}.
\]

\end{thm}

\begin{proof}
We have
\[
\ho{3}{\spl{2}{F}}{\zhalf{\Z}}_0=\zhalf{\rbl{F}_0}=\aug{F}\zhalf{\rbl{F}}=\sum_{a\in \sq{F}}\epm{a}\zhalf{\rbl{F}}
\]
by Lemma \ref{lem:h3sl20}.

  We denote by $S_{v,-1}$ the specialization map $\rpb{F}\to\redpb{k_v}$ corresponding to the character $\epsilon=-1$ 
(see Remark \ref{rem:chi}). 

Note that if $a\in F^\times$ with $v(a)\equiv 0\pmod{2}$, then for any 
$y\in \rpb{F}$, $S_v(\epm{a}y)=\epm{a}S_v(y)=0$ since 
$\an{a}$ acts trivially on $\pb{k_v}$. 
Thus, for any $a\in F^\times$, $y\in \rpb{F}$, we have $S_{v,-1}(\epm{a}y)=0$ for almost 
all $v$. It follows that for any $x\in \zhalf{\aug{F}\rbl{F}}$, $S_{v,-1}(x)=0$ for almost all $v\in\mathcal{V}$. 

Thus the maps $S_{v,-1}$ induce a well-defined $\zhalf{\sgr{F}}$-homomorphism
\[
S:\zhalf{\rbl{F}}_0\to\oplus_{v}\zhalf{\redpb{k_v}},\qquad  x\mapsto \{ S_{v,-1}(x)\}_v.
\]

Finally, let $y=\{ y_v\}_v\in \oplus_v\zhalf{\redpb{k_v}}$.  
Let $T=\{ v\in\mathcal{V}\ |\ y_v\not= 0\}$. For each $v\in T$, choose $x_v\in \zhalf{\rbl{F}}$ with $S_{v,-1}(x_v)=y_{v}$. 

For each $v\in T$, choose $\pi_v\in F^\times$ satisfying $w(\pi_v)\equiv\delta_{w,v}\pmod{2}$ for all $w\in\mathcal{V}$. 
Then 
\[
y=S\left(\sum_{v\in T}\epm{\pi_v}\cdot x_v\right).
\] 
\end{proof}

Since $\kl{k}$ is a $2$-torsion group, we have 
$\zhalf{\redpb{k}}=\zhalf{\pb{k}}$ for a finite field $k$, and we deduce:

\begin{cor}\label{exa:global}
Let $F$ be a global field and let $S$ be a finite set of places of $F$. 
Let 
\[
\mathcal{O}_S=\{ a\in F^\times\ |\ v(a)\geq 0\mbox{ for all } v\not\in S\}\cup \{ 0\}.
\]
Suppose  that $\modtwo{\cl{\mathcal{O}_S}}=0$.
Let $\mathcal{V}$ be the set of all finite places of $F$ not in $S$. 
Then there is  
a surjective homomorphism of $\sgr{F}$-modules
\[
\ho{3}{\spl{2}{F}}{\zhalf{\Z}}_0\to \bigoplus_{v\in\mathcal{V}}\zhalf{\pb{k_v}}.
\]
 
Here, if $\pi_v$ is a uniformizer for $v$, then the square 
class of $\pi_v$ acts as $-1$ on the factor $\zhalf{\pb{k_v}}$ on the right. 
\end{cor}

%Of course, if the field $k$ is finite of order $q$, then $\zhalf{\rrrbl{k}}$ is cyclic of order $\zzhalf{(q+1)}$ or 
%$\zzhalf{(q+1)}/3$ by [include reference] (where $\zzhalf{a}$ denotes the odd part of the integer $a$).

\begin{rem} 
Note that it follows that the groups $\ho{3}{\spl{2}{F}}{\Z}$ are not finitely generated, for any global field $F$. 
By contrast, when $F$ is a global field, 
the groups $\ho{3}{\spl{3}{F}}{\Z}\cong\kind{F}$ are well-known to be finitely generated. 
Furthermore the stabilization map 
$\ho{3}{\spl{2}{F}}{\Z}\to\ho{3}{\spl{3}{F}}{\Z}$ is known to be surjective in this case, by the 
results of \cite{hutchinson:tao2}.
\end{rem}

\begin{rem} Observe that in the case $F=\Q$ the short exact sequence 
\[
0\to \ho{3}{\spl{2}{\Q}}{\zhalf{\Z}}_0\to \ho{3}{\spl{2}{\Q}}{\zhalf{\Z}} \to \zhalf{\kind{\Q}}\to 0
\]
is split, since the composite 
\[
\ho{3}{\spl{2}{\Z}}{\zhalf{\Z}}\to \ho{3}{\spl{2}{\Q}}{\zhalf{\Z}}\to \zhalf{\kind{\Q}}=\zhalf{\bl{\Q}}
=\Z\cdot \cconst{\Q}
\]
is an isomorphism (of groups of order $3$). Thus there is  a natural surjective homomorphism 
\[
\xymatrix{
\ho{3}{\spl{2}{\Q}}{\zhalf{\Z}}\ar@{>>}[r]
&
\zhalf{\kind{\Q}}\oplus\left( \bigoplus_{p}\zhalf{\pb{\F{p}}}\right)\\
}
\]
\end{rem}
When $k$ is algebraically closed,  $\redpb{k}=\pb{k}=\rpb{k}$ and $\pb{k}$ is 
uniquely divisible -- i.e. a $\Q$-vector space -- by the results of Suslin (\cite{sus:bloch}).

\begin{cor}\label{exa:algv}
Let $k$ be an algebraically closed field and let $X$ be an algebraic variety  over $k$, with 
rational function field $k(X)$. Suppose that $X$ is regular in codimension $1$ and that $\modtwo{\cl{X}}=0$.
Let $V$ be the set of codimension $1$ subvarieties.

The theorem says that there is a surjective homomorphism of $\sgr{k(X)}$-modules
\[
\ho{3}{\spl{2}{k(X)}}{\Q}_0\to \oplus_{v\in V}{\pb{k}}
\]
\end{cor}

\begin{rem} Note that this shows that $\ho{3}{\spl{2}{F}}{\Z}_0$ may have a large torsion-free part. 
\end{rem}

For any global field $F$ we have 
\[
\ho{3}{\spl{2}{F}}{\Z}=\varinjlim_{S}\ho{3}{\spl{2}{\mathcal{O}_S}}{\Z}
\]
(where $S$ varies over all finite sets of places). It is natural to inquire to what extent the results described above 
for fields may extend to rings of $S$-integers $\mathcal{O}_S$ (at least when $S$ is large). For example, it follows from 
the equation just quoted that the square classes of $\mathcal{O}_S^\times$ must act nontrivially on the group  
$\ho{3}{\spl{2}{\mathcal{O}_S}}{\Z}$ for large $S$.

As noted in the introduction, there are very few explicit calculations of the homology or cohomology of $S$-arithmetic 
groups such as $\spl{2}{\mathcal{O}_S}$, in spite of  their great interest in geometry and number theory. 
(Over number fields, the homology groups are known to be finitely-generated 
(\cite{borel:serre}), and 
calculation of homology and cohomology are essentially equivalent by the universal coefficient theorem.) For example,
the cohomology of the groups $\spl{2}{\Z}$ and $\spl{2}{\Z[1/p]}$ ($p$ a prime) are known (see \cite{adem:naffah}) 
and there has been extensive work on the calculation of the integral cohomology of the groups $\spl{2}{\mathcal{O}}$ 
where $\mathcal{O}$ is the ring of integers in an imaginary quadratic field (see \cite{rahm:fuchs}, 
\cite{rahm:bianchi},\cite{sengun:bianchi}). 
Little is known, on the other hand,  
about the integral (co)homology of $\spl{2}{\Z[1/N]}$ when $N$ has at least $2$  prime factors 
(see 
\cite{hesselmann} and \cite{kuhnlein} for some partial results). 
%Recently there has been much concentration on the cohomology of the Bianchi groups.  

Here we give an example of how the module structure and specialization maps developed above 
can allow us to construct nontrivial homology classes in the groups $\ho{3}{\spl{2}{\mathcal{O}_S}}{\Z}$:

Let 
\[
t:=\matr{-1}{1}{-1}{0}\in \spl{2}{\Z}
\]
and let $G$ be the cyclic group of order $3$ generated by $t$.
For any ring $A$, we will let  $\one:=\one_A$ 
denote the image of the generator of $\ho{3}{G}{\Z}$ in $\ho{3}{\spl{2}{A}}{\Z}$ (see Remark \ref{rem:h3sl2z} above). 

Thus if $A$ is a subring of the field $F$, 
$\one_A$ maps to $-\cconst{F}$ under the map $\ho{3}{\spl{2}{A}}{\Z}\to\rbl{F}$.  

%For an abelian group $A$, and for $n$ a positive integer, we let $\ptor{A}{n}$ denote $\{ a\in A\ |na=0\}$.

\begin{thm} 
\label{thm:3tors}
Let $F$ be a number field. Let $S$ be a nonempty set of finite primes of $F$ such that
$\modtwo{\cls{S}{F}}=0$, where $\cls{S}{F}$ denotes the subgroup of $\cl{\mathcal{O}_F}$ generated by the primes 
of $S$. 

Then there is a natural surjective map of $\F{3}[\modtwo{\mathcal{O}_S^\times}]$-modules 
\[
\ptor{\ho{3}{\spl{2}{\mathcal{O}_S}}{\Z}}{3}\to \oplus_{v\in S}\ptor{\pb{k_v}}{3}.
\]
\end{thm}
\begin{proof} For $v\in S$, let  $R_v$ denote the composite
\[
\xymatrix{
\ho{3}{\spl{2}{\mathcal{O}_S}}{\Z}\ar[r]
&\ho{3}{\spl{2}{F}}{\Z}\ar[r]
&\rbl{F}\ar[r]^-{S_{v,-1}}
& \redpb{k_v}.}
\]

Observe that the associated action of $\mathcal{O}_S^\times$ on $\redpb{k_v}$ is decribed as follows: For 
$a\in \mathcal{O}_S^\times$, $\an{a}$ acts as multiplication by $(-1)^{v(a)}$. 

By \cite{hut:arxivcplx11}, section 7, $\ptor{\redpb{k}}{3}=\ptor{\pb{k}}{3}$ is generated by 
$\cconst{k}=2\bconst{k}$ for any finite field $k$. 

Our hypothesis on $\cls{S}{F}$ guarantees that the map
\[
\mathcal{O}_S^\times\to\oplus_{v\in S}\Z/2 
\]
is surjective, and thus for each $w\in S$, there exists $u_w\in \mathcal{O}_S^\times$ satisfying 
$v(u_w)\equiv\delta_{v,w}\pmod{2}$ for all $v\in S$. 

Finally, for each $w\in S$,  let 
\[
\one_{w,S}:=\epm{u_w}\cdot\one_{\mathcal{O}_S}\in \ptor{\ho{3}{\spl{2}{\mathcal{O}_S}}{\Z}}{3}
\]
(see Remark \ref{rem:h3sl2z}) and hence  
\[
R_v(\one_{w,S})=-\epm{u_w}\cdot \cconst{k_v}=-\delta_{v,w}\cconst{k_v}.
\]
\end{proof}. 

\begin{cor} Let $F$ be a number field not containing $\zeta_3$. Let $S$ be a finite set of primes of $F$ for which 
$\modtwo{\cls{S}{F}}=0$. Let $r(S)=\card{\{v\in S\ |\ \card{k_v}\equiv -1\pmod{3}\}}$. Then
\[
3\mbox{-rank}\left(\ho{3}{\spl{2}{\mathcal{O}_S}}{\Z}\right)\geq 1+r(S). 
\]
\end{cor}
\begin{proof}
The $3$-rank of $\oplus_{v\in S}\ptor{\pb{k_v}}{3}$ is $r(S)$ 
since $\ptor{\pb{k_v}}{3}=\ptor{\bl{k_v}}{3}$, and 
$\bl{k_v}$ is cyclic of order $(\card{k_v}+1)/2$ or $\card{k_v}+1$  by the results of 
\cite{hut:arxivcplx11}. 

On the other hand, letting $u_w$ be as in the proof of Theorem \ref{thm:3tors}, the element 
\[
\Bbb{D}_+:=\left(\prod_{w\in S}\ep{u_w}\right)\one_{\mathcal{O}_S}\in \ptor{\ho{3}{\spl{2}{\mathcal{O}_S}}{\Z}}{3}
\]
lies in the kernel of the map 
\[
\oplus_vR_v:\ptor{\ho{3}{\spl{2}{\mathcal{O}_S}}{\Z}}{3}\to \oplus_{v\in S}\ptor{\bl{k_v}}{3}
\]

However,  for any place $v$ of the field $F$, the image of $\Bbb{D}_+$ in $\rbl{F}$ maps to 
$-\cconst{k_v}$ under the specialization map $S_{v,1}:\rbl{F}\to \redpb{k_v}$ corresponding to the trivial character 
$\chi=1$. If we choose $v$ for which $3$ divides $\card{k_v}+1$ then
$\cconst{k_v}$ has order $3$. It follows that $\Bbb{D}_+$ has order $3$. 

%On the other hand, for any $v\in S$, 
%$S_{v,-1}(\Bbb{D}_+)=S_{v,-1}(\epm{u_v}\Bbb{D}_+)=S_{v,-1}(0)=0$. Thus $\Bbb{D}_+$ lies in the kernel of the 
%surjective homomorphism from $ \ho{3}{\spl{2}{\mathcal{O}_S}}{\Z}$ to $\oplus_{v\in S}\ptor{\pb{k_v}}{3}$.
\end{proof}

\section{The third homology of $\mathrm{SL}_2$ of local fields}\label{sec:local}
In this section, we use the properties of the refined Bloch group to calculate $\ho{3}{\spl{2}{F}}{\Z}$ up 
to $2$-torsion 
 when $F$ is a local field -- i.e. is complete with respect to discrete value -- 
with finite residue field of odd characteristic (Theorem \ref{thm:local}).

\subsection{Preliminary results}\label{sec:localprelim}
We will begin by recalling some of the relevant facts about structure of the multiplicative group of a local field.

Let $F$ be a field with discrete value $v$ and residue field $k$. Let $\pi$ be a uniformizer. Then the homomorphism
$v:F^\times\to \Z$ splits and we have
\[
F^\times=U\cdot \pi^{\Z}\cong U\times \Z.
\]
Thus 
\[
\sq{F}\cong U/U^2\times \modtwo{\Z}.
\]

Now suppose that $F$ is complete with respect to  $v$ and that 
$2\in U$ (i.e. $2\not= 0$ in $k$). Then Hensel's Lemma (\cite[Chapter II, 4.6]{neukirch:ant}) implies that, 
for any $u\in U$, if $\bar{u}\in (k^\times)^2$ then $u\in U^2$. In particular, in this case $U_1=U_1^2$ and  
 $U/U^2\cong k^\times/(k^\times)^2=\sq{k}$. 

It follows that  if $F$ is complete with respect to $v$ and if $k$ is finite of odd order, then 
\[
\sq{F}\cong \sq{k}\times \Z/2
\]   
has order $4$ with generators $\an{\pi}$ and $\an{u}$ where $u$ is any nonsquare unit.

Suppose that $L/F$ is an extension of local fields. Let $\pi$ be a uniformizer of $F$. Then there is an induced extension
of residue fields $k_L/k_F$ and $[L:F]=v_L(\pi)[k_L:k_F]$ for any uniformizer $\pi$ of $F$ 
(\cite[Chapter II, 6.8]{neukirch:ant}).  

An extension $L/F$ of local fields is said to be \emph{unramified} if a uniformizer $\pi$ of $F$ is also a uniformizer in
$L$. For example, if $\zeta_n$ is a primitive $n$-th root of unity in some separable closure of $F$ and if 
$(n,\mathrm{char}(k))=1$, then $F(\zeta_n)/F$ is an unramified extension 
(\cite[Chapter II, Proposition 7.12]{neukirch:ant}).
Furthermore, Hensel's Lemma  
implies that, when $(n,\mathrm{char}(k))=1$, $\zeta_n\in F$ if and only if $k$ contains a primitive $n$th root of unity.

Local class field theory gives us valuable information about norm groups $N_{L/F}(L^\times)$ of finite algebraic extensions
$L/F$ of local fields:
\begin{enumerate}
\item If $L/F$ is a finite abelian extension of local fields of degree $d$, then $F^\times/N_{L/F}(L^\times)$ is a group 
of order $d$ (\cite[Chapter 5, Theorem 1.3]{neukirch:ant}).
\item If $L/F$ is a finite unramified extension of local fields, then $N_{L/F}(U_L)=U_F$ 
(\cite[Chapter V, Corollary 1.2]{neukirch:ant}).
\item The theory of the Hilbert symbol tells us that if $(n,\mathrm{char}(k))=1$ and if $\zeta_n\in F$, 
then $u\in U_F$ is a norm from $F(\sqrt[n]{b})$ if and only if 
\[
\bar{u}^{v(b)\frac{\card{k^\times}}{n}}=1\mbox{ in } k^\times.
\]
(See \cite[Chapter V, 3.2 (iii) and 3.4]{neukirch:ant}.)
\end{enumerate} 

As above, $E$ denotes the field obtained from $F$ by adjoining a root of $\Phi(T)=T^2-T+1$.  
\begin{lem} \label{lem:localconst}
Let $F$ be a local field with finite residue field $k$ of order $q=p^f$. Suppose that $\Q_3\not\subset F$.  
Then
\begin{enumerate}
\item $\pf{a}\cconst{F}=0$ if and only if $a\in \an{-1}\cdot N_{E/F}(E^\times)$.
\item For any uniformizer, $\pi$, of $F$ the specialization map $S_{\pi,-1}:\rpb{F}\to\redpb{k}$ induces an isomorphism
of $\sgr{F}$-modules
\[
\aug{F}\cdot\cconst{F}\cong \Z\cdot\cconst{k} =\ptor{\pb{k}}{3}.
\]
\end{enumerate} 
\end{lem}

\begin{proof}
We have already seen that if $a\in \an{-1}\cdot N_{E/F}(E^\times)$ then $\pf{a}\cconst{F}=0$ (Theorem \ref{thm:df}). 
In other words, the action of $\sgr{F}$ on $\cconst{F}$ factors through the group ring 
$\Z[F^\times/\an{-1}\cdot N_{E/F}(E^\times)]$.

If $\mathrm{char}(F)=3$, then $E=F$ and hence $\cconst{F}=0$
by Corollary \ref{cor:bconst} (2). Similarly, $\cconst{k}=0$.  So the result 
holds trivially. 

So we may assume that $\mathrm{char}(F),\mathrm{char}(k)\not=3$. In this case $E=F(\sqrt{-3})=F(\zeta_3)$.
By the facts just quoted, $E/F$ is unramified and every unit is a norm. 
In particular, $-1\in  N_{E/F}(E^\times)$. 

If $\zeta_3\in F$, then 
$\pf{a}\cconst{F}=0$ for all $a$, and also $\cconst{k}=0$. The result holds again for trivial reasons. 

Otherwise, $\zeta_3\not\in F$ and $F^\times/N_{E/F}(E^\times)$ is cyclic of order $2$ generated by the class of a 
uniformizer $\pi$. 
It follows that 
\[
\aug{F}\cconst{F}=\Z\pf{\pi}\cconst{F}
\]
has order $1$ or $3$. 
However,  since $\zeta_3\not\in k$, $q\not\equiv 1\pmod{3}$ so $3|q+1$ and $\cconst{k}$ has order $3$. 
$S_{\pi,-1}(\cconst{F})=\cconst{k}\not= 0$ 
and since $\pi$ acts as $-1$ on the right, $S_{\pi,-1}(\pf{\pi}\cconst{F})=-2\cconst{k}=\cconst{k}\not =0$, so that 
the statements of the lemma follow immediately. 
\end{proof}

\begin{lem}\label{lem:q3} Suppose that $\Q_3\subset F$. Then $-1\in N_{E/F}(E^\times)$ if and only if $[F:\Q_3]$ is 
even. 
\end{lem}

\begin{proof} 
By (3) above (with $n=2$ and $u=-3$), 
\begin{eqnarray*}
-1\in N_{E/F}(E^\times)&\iff & (-1)^{v(3)\frac{3^f-1}{2}}=1\\
&\iff & v(3)\frac{3^f-1}{2}\mbox{ is even}\\
&\iff & v(3)f\mbox{ is even}.
\end{eqnarray*}
However, by the result on the degree of local field extensions quoted before Lemma \ref{lem:localconst} above, 
$[F:\Q_3]=v(3)f$. 
\end{proof}
\begin{cor}\label{cor:q3}
If $[F:\Q_3]$ is odd then 
\[
\aug{F}\cconst{F}=0=\Z\cdot \cconst{k}.
\]
\end{cor}
\begin{proof}
Under these hypotheses, $-1\not\in N_{E/F}(E^\times)$ and thus $F^\times=\an{-1}\cdot N_{E/F}(E^\times)$ (since 
$[F^\times:N_{E/F}(E^\times)]=2$).  Thus $\aug{F}\cconst{F}=0$ by Theorem \ref{thm:df}. But $\mathrm{char}(k)=3$ 
and hence $\cconst{k}=0$ also.
\end{proof}

\begin{rem}\label{rem:q3} 
Of course, if $\Q_3\subset F$ and $-1\in N_{E/F}(E^\times)$ then our results do not rule out  the 
possibility that $\aug{F}\cconst{F}$ has order $3$ (rather than $1$), but this module, if nontrivial, 
 cannot be detected by $S_{\pi,-1}$ (since $\cconst{k}=0$ in this case).
\end{rem}

The following lemma will be central to our computations below:

\begin{lem}\label{cor:ann}
In any field $F$ we have
\[
\epm{b}\gpb{a}=0\mbox{ in } \zhalf{\rrrpb{F}}
\]
 whenever $a,b\in F^\times$ with $a\equiv -b \pmod{(F^\times)^2}$.
\end{lem}
\begin{proof} 
Let $a\in F^\times$, $a\not= 1$. 

Then $\an{a}\gpb{a}=\an{-1}\gpb{a}$ in $\rrrpb{F}$ by Lemma \ref{lem:rrrpb}. Thus $\an{-a}\gpb{a}=\gpb{a}$ and hence 
$\pf{-a}\gpb{a}=0$. It follows that $\epm{-a}\gpb{a}=0$ in $\zhalf{\rrrpb{F}}$. 

Of course, $\epm{x}$ only depends on the square class of $x$, so the result follows. 
\end{proof}

For the remainder of this section we let $F$ be a local field complete with respect to the discrete valuation $v$, 
with residue 
field $k$ of order $q=p^f$, with $p$ \emph{odd}. Furthermore, for simplicity, we will suppose that if $\Q_3\subset F$ 
then $[F:\Q_3]$ is odd.
  
Recall that  $\sq{F}$ has order $4$: if 
$\pi$ is a uniformizing parameter and if $u$ is a nonsquare unit $\sq{F}=\{ 1,\an{\pi},\an{u},\an{u\pi}\}$. 
(If $q\equiv 3\pmod{4}$, we can take $u=-1$.)

We  let $\widehat{\sq{F}}$ denote the group of characters $\hom{}{\sq{F}}{\mu_2}=\hom{}{\sq{F}}{\{\pm 1\} }$ of 
$\sq{F}$. 
Given $\chi\in 
\widehat{\sq{F}}$, we have the associated idempotent
\[
\idem{\chi}=\frac{1}{4}\sum_{g\in \sq{F}}\chi(g)\an{g}.
\]
Observe that if $\chi$ is a \emph{nontrivial} character on $\sq{F}$, then
\[
\idem{\chi}=\prod_{a\in\chi^{-1}(-1)}\epm{a}\in \zhalf{\sgr{F}}.
\]

If $M$ is an $\zhalf{\sgr{F}}$-module and if $\chi\in\widehat{\sq{F}}$, then $M_\chi:=\idem{\chi}(M)$ is a submodule and 
$g\cdot m=\chi(g)m$ for all $g\in \sq{F}$. Observe that the functor $M\to M_\chi$ is an exact functor on the category of 
$\zhalf{\sgr{F}}$-modules.  

For the rest of this section we fix the following: Let $\pi$ be a uniformizing parameter for $F$ and let $u$ be a 
fixed nonsquare 
unit, which we take to be $-1$ in the case $q\equiv 3\pmod{4}$.  Clearly a nontrivial 
character in $\widehat{\sq{F}}$ is determined by $\chi^{-1}(-1)$.
We label the four characters as follows: $\chi_1$ is the trivial character, $\chi$ is the character with 
$\chi^{-1}(-1)=\{\an{\pi},\an{u\pi}\}$, $\psi$ is the character with $\psi^{-1}(-1)=\{\an{u\pi},\an{u}\}$ and $\psi'$ is 
the remaining character.

Note that for any $\zhalf{\sgr{F}}$-module $M$,  we have a decomposition
\[
M=M_{\chi_1}\oplus M_\chi\oplus M_\psi\oplus M_{\psi'}
\]
where 
\[
M_{\chi_1}=M^{\sq{F}}\cong M_{\sq{F}}=\ho{0}{F^\times}{M}
\]
and
\[
\aug{F}M=M_\chi\oplus M_\psi\oplus M_{\psi'}.
\]

Since $\Tor{\mu_F}{\mu_F}$ is a trivial $\sq{F}$-module, from the exact sequence
\[
0\to\zhalf{\Tor{\mu_F}{\mu_F}}\to\ho{3}{\spl{2}{F}}{\zhalf{\Z}}\to\zhalf{\rbl{F}}\to 0
\]
it follows that 
\[
\ho{3}{\spl{2}{F}}{\zhalf{\Z}}_\rho=\zhalf{\rbl{F}}_\rho
\]
for any $\rho\not=\chi_1$. 

On the other hand, 
\[
\ho{3}{\spl{2}{F}}{\zhalf{\Z}}_{\chi_1}=\ho{0}{F^\times}{\ho{3}{\spl{2}{F}}{\zhalf{\Z}}}=\zhalf{\kind{F}}.
\]

Thus we have a $\zhalf{\sgr{F}}$-module decomposition
\[
\ho{3}{\spl{2}{F}}{\zhalf{\Z}}\cong \zhalf{\kind{F}}\oplus\zhalf{\rbl{F}}_{\chi}\oplus\zhalf{\rbl{F}}_{\psi}
\oplus\zhalf{\rbl{F}}_{\psi'}.
\]
Our goal in the remainder of this section is to show that the last two factors on the right are zero and that 
the second factor is isomorphic, via $S_{\pi,-1}$, to $\zhalf{\pb{k}}$.

In order to do this, we make a couple of reductions:

\begin{lem}\label{lem:localdf} Let $F$ be as stated. Then 
\begin{enumerate}
\item $(\cconstmod{F})_\psi=0=(\cconstmod{F})_{\psi'}$.
\item $(\cconstmod{F})_\chi=\aug{F}\cconst{F}$.
\end{enumerate} 
\end{lem}

\begin{proof} Note that 
\[
\aug{F}\cconst{F}=(\cconstmod{F})_\chi\oplus(\cconstmod{F})_\psi\oplus(\cconstmod{F})_{\psi'}.
\]
So we must prove that the last two factors are $0$.

Recall that $\idem{\psi}=\epm{u}\epm{u\pi}$ and $\idem{\psi'}=\epm{u}\epm{\pi}$.  
Now $\epm{u}=-\pf{u}/2$ and our conditions guarantee that $u\in \an{-1}N_{E/F}(E^\times)$ so that $\pf{u}\cconst{F}=0$
by Theorem \ref{thm:df}. 
\end{proof}

In view of Corollary \ref{cor:rrrbl} and Lemma \ref{lem:rrrrbl} we have:

\begin{cor} \label{cor:psipsi}
$\zhalf{\rbl{F}}_{\psi}=\zhalf{\rrrrbl{F}}_{\psi}$ and $\zhalf{\rbl{F}}_{\psi'}=\zhalf{\rrrrbl{F}}_{\psi'}$.
\end{cor}

Observe that $S_{\pi,-1}$ induces a well-defined surjective homomorphism 
\[
\bar{S}_{\pi,-1}:\rrrrpb{F}\to \rredpb{k}. 
\]

\begin{lem} \label{lem:rbfbar}
The homomorphism
\[
S_{\pi,-1}:\zhalf{\rbl{F}}_{\chi}\to \zhalf{\redpb{k}}
\]
is an isomorphism if and only if the homomorphism 
\[
\bar{S}_{\pi,-1}:\zhalf{\rrrrbl{F}}_\chi\to \zhalf{\rredpb{k}}
\]
is an isomorphism. 
\end{lem}

\begin{proof} Using Lemma \ref{lem:localdf} (2), we have a commutative diagram of $\sgr{F}$-modules with exact rows
\[
\xymatrix{
0\ar[r]
&
\aug{F}\cconst{F}\ar[r]\ar[d]^{\cong}
&
\zhalf{\rbl{F}}_{\chi}\ar[r]\ar[d]^{S_{\pi,-1}}
&
\zhalf{\rrrrbl{F}}_{\chi}\ar[r]\ar[d]^{\bar{S}_{\pi,-1}}
&0\\
0\ar[r]
&
\Z\cconst{k}\ar[r]
&
\zhalf{\redpb{k}}\ar[r]
&
\zhalf{\rredpb{k}}\ar[r]
&
0
}
\]
in which the left vertical arrow is an isomorphism by Lemma \ref{lem:localconst} (2). 
\end{proof}

\subsection{$S_3$-dynamics}
Again, in this section, $F$ denotes a field complete with respect to a discrete value $v$ with finite 
residue field $k$ of odd order $q$.

For any field $L$, the transformations 
\[
\sigma,\tau:\projl{L}\to\projl{L},\quad \sigma(x)=x^{-1}\mbox{ and }\tau(x)=1-x
\]
 determine an action 
of the symmetric group on $3$ letters  $S_3$.  $\{ 0,1, \infty\}$ is an orbit for this action, and hence 
$S_3$ acts on $\projl{L}\setminus\{ 0,1,\infty\}=F^\times\setminus\{ 1\}$. 

\begin{lem}\label{lem:s3}
If $\alpha\in \sgr{F}$ and if $\alpha\gpb{x}=0$ in $\rrrrpb{F}$. Then $\alpha\gpb{\gamma(x)}=0$ for all 
$\gamma\in S_3$.
\end{lem}
\begin{proof}
In $\rrrrpb{F}$ we have $\cconst{F}=0$ and $\suss{1}{x}=0$ for all $x$, so that 
\[
\gpb{x^{-1}}=\gpb{1-x}=-\an{-1}\gpb{x}.
\]
It follows that if $\alpha\gpb{x}=0$ for some $\alpha\in \sgr{F}$, then $\alpha\gpb{\sigma(x)}=\alpha\gpb{\tau(x)}=0$ 
also. The statement follows since $\sigma$ and $\tau$ generate $S_3$
\end{proof}

\begin{rem} We will apply this below in the case where $\rho$ is a character of $\sq{F}$ and 
$\alpha=\idem{\rho}\in\zhalf{\sgr{F}}$
\end{rem}

Let $\bar{R}=(k^\times)^2$ and let $\bar{N}=k^\times\setminus (k^\times)^2$. 
Let 
\[
\bar{R}_1=\{ a\in \bar{R}\setminus\{ 1\} | \tau(a)\in \bar{R}\},\qquad 
\bar{R}_{-1}=\{a\in \bar{R} | \tau(a)\in \bar{N}\}.
\]

Let
\[
\bar{N}_1=\{ a\in \bar{N} | \tau(a)\in \bar{R}\},\qquad 
\bar{N}_{-1}=\{a\in \bar{N} | \tau(a)\in \bar{N}\}.
\]

Thus we have a partition 
\[
k^\times=\{ 1\}\cup \bar{R}_1\cup \bar{R}_{-1}\cup\bar{N}_1\cup\bar{N}_{-1}
\]

Let $U$ be the group of units of $F$ and let $U_1$ be the kernel of the reduction map $U\to k^\times$.
Then there is a corresponding partition
\[
U=R\cup N = U_1\cup R_1\cup R_{-1}\cup N_1\cup N_{-1}
\]
where $R:=U^2=\{ a\in U | \bar{a}\in \bar{R}\}$, $N= U\setminus R$, 
$R_1:=\{ a\in U | \bar{a}\in\bar{R}_1\}=\{ a \in R | \tau(a)\in R\}$, $R_{-1}:=\{ a\in U | \bar{a}\in\bar{R}_{-1}\}$,
$N_1:=\{ a\in U | \bar{a}\in\bar{N}_1\}$ and $N_{-1}:=\{ a\in U | \bar{a}\in\bar{N}_{-1}\}$.

\begin{lem}\label{lem:rn}
%Let $k$ be a finite field with $q$ elements.
\begin{enumerate}
\item If $q\equiv 1\pmod{4}$ then 
\begin{eqnarray*}
&\bar{R}_1=\sigma(\bar{R}_1)=\tau(\bar{R}_1)\\
&\sigma(\bar{N}_1)=\bar{N}_{-1}\mbox{ and } \sigma(\bar{R}_{-1})=\bar{R}_{-1}
\end{eqnarray*}
Furthermore
\[
|\bar{R}_{-1}|=|\bar{N}_{-1}|=|\bar{N}_1|=\frac{q-1}{4}\mbox{ and } |\bar{R}_1|=\frac{q-5}{4}.
\]
\item If $q\equiv 3\pmod{4}$ then 
\begin{eqnarray*}
&\bar{N}_{-1}=\sigma(\bar{N}_{-1})=\tau(\bar{N}_{-1})\\
&\sigma(\bar{R}_1)=\bar{R}_{-1}\mbox{ and } \sigma(\bar{N}_{1})=\bar{N}_{1}
\end{eqnarray*}
Furthermore
\[
|\bar{R}_{-1}|=|\bar{N}_{1}|=|\bar{R}_1|=\frac{q-3}{4}\mbox{ and } |\bar{N}_{-1}|=\frac{q+1}{4}.
\]

\end{enumerate}
\end{lem}

\begin{proof}
\begin{enumerate}
\item Clearly $\tau(\bar{R}_1)=\bar{R}_1$ by definition.  Furthermore, if $a\in \bar{R}_1$ then 
\[
1-\frac{1}{a}=(-1)\cdot(1-a)\cdot\frac{1}{a}\in \bar{R}
\]
and thus $a^{-1}=\sigma(a)\in \bar{R}_1$. 

It follows that $\sigma(\bar{R}_{-1})=\bar{R}_{-1}$ since $\sigma(\bar{R})=\bar{R}$. 

Similarly, if $a\in\bar{N}_{1}$ then $a^{-1}\in \bar{N}$ $1-a\in\bar{R}$ and thus 
\[
1-\frac{1}{a}=(-1)\cdot (1-a)\cdot \frac{1}{a}\in N
\] 
so that $\sigma(\bar{N}_{1})=\bar{N}_{-1}$.

The second statement follows by simple counting since 
\[
|\bar{N}_1|+|\bar{N}_{-1}|=|\bar{N}|=(q-1)/2.
\]

\item The argument is similar except that this time 
\[
a-1=(-1)\cdot(1-a)\in
\left\{
\begin{array}{ll}
\bar{N},& 1-a\in \bar{R}\\
\bar{R},& 1-a\in \bar{N}\\
\end{array}
\right.
\]
since $-1\in \bar{N}$.
\end{enumerate}
\end{proof}

Taking inverse images in $U\subset F$, and noting that $U\setminus U_1$ is closed under the action of $S_3$, we obtain :
\begin{cor} \label{cor:rn}
%Let $F$ be a local field with finite residue field of order $q$.
\begin{enumerate}
\item If $q\equiv 1\pmod{4}$ then 
\begin{eqnarray*}
&{R}_1=\sigma({R}_1)=\tau({R}_1)\\
&\sigma({N}_1)={N}_{-1}\mbox{ and } \sigma({R}_{-1})={R}_{-1}
\end{eqnarray*}
\item If $q\equiv 3\pmod{4}$ then 
\begin{eqnarray*}
&{N}_{-1}=\sigma({N}_{-1})=\tau({N}_{-1})\\
&\sigma({R}_1)={R}_{-1}\mbox{ and } \sigma({N}_{1})={N}_{1}
\end{eqnarray*}
\end{enumerate}

\end{cor}

We will need the following elementary result below:
\begin{lem}\label{lem:el}
Let $k$ be a finite field of odd order, and let $n\in\bar{N}$. Then there exist $r_1,r_2\in\bar{R}$ such that 
$n=r_1+r_2$.
\end{lem}
\begin{proof}
If $q\equiv 1\pmod{4}$, there exists $r\in \bar{R}$ with $rn\in \bar{N}_{1}$. Thus $1-rn=s\in \bar{R}$ and 
\[
n=\frac{1}{r}+\left(-\frac{s}{r}\right)=r_1+r_2.
\]

If $q\equiv 3\pmod{4}$ then there exists $r\in \bar{R}$ with $rn\in \bar{N}_{-1}$. Then $1-nr=m\in \bar{N}$ and thus  
\[
n=\frac{1}{r}+\left(-\frac{m}{r}\right)=r_1+r_2.
\]
\end{proof}

\subsection{The main result}
Let $F$ be a local field with discrete value $v$ and finite residue field $k$ of odd order. Throughout $\pi$ 
will denote a choice of uniformizer.

Let $M(F)$ denote the $\zhalf{\sgr{F}}$-module $\epm{-1}\zhalf{\aug{F}}$.

\begin{lem}\label{lem:mf}

\begin{enumerate}  
\item If $\rho\not=\chi_1$ then 
there is a natural short exact sequence of $\zhalf{\sgr{F}}$-modules
\[
\xymatrix{
0\ar[r]
&
\zhalf{\rrrrbl{F}}_\rho\ar[r]
&
\zhalf{\rrrrpb{F}}_\rho\ar[r]^{\bar{\lambda}_1}
&
M(F)_\rho\ar[r]
&
0.
}
\]
\item If $q\equiv 1\pmod{4}$ and if $\rho\not=\chi_1$ then
\[
\zhalf{\rrrrbl{F}}_\rho=\zhalf{\rrrrpb{F}}_\rho.
\]
\item If $q\equiv 3\pmod{4}$ then 
\[
\zhalf{\rrrrbl{F}}_\chi=\zhalf{\rrrrpb{F}}_\chi.
\]
\end{enumerate}
 \end{lem}

\begin{proof}
 There is  an exact sequence
\[
\xymatrix{
0\ar[r]
&
\rrrrbl{F}\ar[r]
&
\rrrrpb{F}\ar[r]^{\Lambda}
&
\rrasym{2}{\Z}{F^\times}
}
\]  

%where
%\[
%\rrasym{2}{\Z}{F^\times}=\frac{\rasym{2}{\Z}{F^\times}}{\an{\rsf{x}{-x}}}.
%\]

The kernel of the surjection $\rasym{2}{\Z}{F^\times}\to\aug{F}^2$ is a trivial $\sgr{F}$-module. Furthermore, since 
$\aug{F}/\aug{F}^2\cong\sq{F}$ is a $2$-torsion group, we have $\zhalf{\aug{F}^2}=\zhalf{\aug{F}}$. Putting these 
facts together we obtain an isomorphism of $\sgr{F}$-modules 
\[
\zhalf{\rasym{2}{\Z}{F^\times}}_\rho\cong \zhalf{\aug{F}}_\rho
\]
for all $\rho\not=\chi_1$. This in turn induces an isomorphism 
\[
\zhalf{\rrasym{2}{\Z}{F^\times}}_\rho\cong \left(\frac{\zhalf{\aug{F}}}{\an{\pf{x}\pf{-x}}}\right)_\rho=
\left(\frac{\zhalf{\aug{F}}}{\pp{-1}\zhalf{\aug{F}}}\right)_\rho=
\left(\frac{\zhalf{\aug{F}}}{\ep{-1}\zhalf{\aug{F}}}\right)_\rho\cong\epm{-1}\zhalf{\aug{F}}_\rho=M(F)_\rho
\]  
and the homomorphism 
\[
\zhalf{\rrrrpb{F}}_\rho\to\zhalf{\rrasym{2}{\Z}{F^\times}}_\rho\cong M(F)_\rho
\]
is induced by the map $\bar{\lambda}_1:\rrrrpb{F}\to M(F)$
\[
\gpb{x}\mapsto \epm{-1}\lambda_1(\gpb{x})=\epm{-1}\pf{x}\pf{1-x}.
\]

To see that this homomorphism is surjective, observe first that if $q\equiv 1\pmod{4}$, then $\an{-1}=1$ and hence 
$\epm{-1}=0$ in $\zhalf{\sgr{F}}$. Thus $M(F)=0$ in this case, and there is nothing to prove. This also implies 
statement (2) of the lemma.

So we can suppose that  $q\equiv 3\pmod{4}$. Then, as a $\zhalf{\Z}$-module, $\zhalf{\aug{F}}$ is free of rank $3$ with 
basis $\epm{-1}$, $\epm{\pi}$ and $\epm{-\pi}$. Then 
$\epm{-1}(\zhalf{\aug{F}})$ is a free $\zhalf{\Z}$-module of rank $2$ with basis $\epm{-1}$ and $\epm{\pi}\epm{-1}$ 
(since, for example, $\epm{-1}\epm{-\pi}=-(\epm{-1}+\epm{\pi}\epm{-1})$.) As 
an $\zhalf{\sgr{F}}$-module, it is thus generated by $\epm{-1}$. 

Now let $n\in N_{-1}$. Then  $1-n\in N_{-1}$ and 
\[
\lambda_1\left(\frac{1}{4}\gpb{n}\right)
= \frac{1}{4}\pf{n}\pf{1-n}
=\frac{1}{4}\pf{-1}\pf{-1}
=(\epm{-1})^2=\epm{-1}
\]
so that the map  is surjective.
 
Finally, observe also that if $q\equiv 3\pmod{4}$ then $\idem{\chi}\epm{-1}=0$ so that $M(F)_\chi=0$ and statement (3)
also follows. 
\end{proof}

We will use the following notation in the remainder of this section: 
Given a character, $\rho$ of $\sq{F}$ and $a\in F^\times$, we will denote $\idem{\rho}(\gpb{a})\in \zhalf{\rrrrpb{F}}$ 
by $\gpb{a}_\rho$.

In the each of next few lemmas, the strategy of proof is largely the same. 
 The proof involves showing that $\gpb{a}_\rho=0$ for $a$ belonging to some subset of $F^\times$. We 
begin with Lemma \ref{cor:ann} to show that $\gpb{a}_\rho=0$ for $a$ belonging to one or more square classes, 
and then use Lemma \ref{lem:s3} to deduce the vanishing 
of $\gpb{a}_\rho$ for a larger set of $a$.  

In order to follow the arguments, it may be helpful for the reader to keep 
the following decompositions in mind: Recall that there are four square classes, or cosets modulo $(F^\times)^2$, in the 
group $F^\times$. Thus
\[
F^\times = (F^\times)^2\cup u\cdot(F^\times)^2\cup \pi\cdot (F^\times)^2\cup u\pi\cdot(F^\times)^2
\]
and 
\[
(F^\times)^2=\bigcup_{n\in\Z}R\pi^{2n}=R\cup\bigcup_{n\not=0}R\pi^{2n}=U_1\cup R_1\cup R_{-1}\cup \bigcup_{n\not=0}R\pi^{2n}.
\]
\[
u\cdot(F^\times)^2=\bigcup_{n\in\Z}N\pi^{2n}=N\cup \bigcup_{n\not= 0}N\pi^{2n}=N_1\cup N_{-1}\cup \bigcup_{n\not= 0}N\pi^{2n}.
\]
\[
\pi\cdot(F^\times)^2=\bigcup_{n\in\Z}R \pi^{2n-1}.
\]
\[
u\pi\cdot (F^\times)^2=\bigcup_{n\in\Z}N \pi^{2n-1}.
\]
\[
U_1=\bigcup_{m\geq 1} (1-U\pi^m)=\bigcup_{n\geq 1}(1-U\pi^{2n})\cup\bigcup_{n\geq 1}(1-U\pi^{2n-1}).
\]

We will  repeatedly use the \emph{key fact} that $U_1=U_1^2$ and thus $1-a$ is a square whenever 
$v(a)>0$.

\begin{lem}\label{lem:prelim}
Let $\rho$ be a character of $\sq{F}$. Then the following are equivalent:
\begin{enumerate}
\item $\gpb{a}_\rho=0$ in $\rrrrpb{F}$ for all $a\in U_1$
\item $\gpb{a}_\rho=0$ in $\rrrrpb{F}$ for all $a$ satisfying $v(a)\not =0$. 
\end{enumerate}
\end{lem}
\begin{proof}
Suppose first that $\gpb{a}_\rho=0$ for all $a\in U_1$. Let $b\in F^\times$ with $v(b)>0$. Then $c:=\tau(b)=1-b\in U_1$. 
Thus $\gpb{c}_\rho=0$ by assumption. It follows that $\gpb{b}_\rho=\gpb{\tau(c)}_\rho=0$ by Lemma \ref{lem:s3} 
with $\alpha=\idem{\rho}$. 

On the other hand, if $b\in F^\times$ with $v(b)<0$, then $v(b^{-1})>0$ and thus $\gpb{b^{-1}}_\rho=0$. But then 
$\gpb{b}_\rho=\gpb{\sigma(b^{-1})}_\rho=0$ by Lemma \ref{lem:s3} again.

Conversely, suppose that $\gpb{a}_\rho=0$ whenever $v(a)\not=0$ and that $b\in U_1$. Then $v(1-b)>0$. Thus 
$\gpb{1-b}_\rho=0$ and hence $\gpb{b}_\rho=\gpb{\tau(1-b)}_\rho=0$.
\end{proof}

\begin{lem}\label{lem:chi}
The $\zhalf{\sgr{F}}$-module homomorphism
\[
S_{\pi,-1}:\zhalf{\rpb{F}}_\chi\to\zhalf{\redpb{k}}
\]
is an isomorphism.
\end{lem}

\begin{proof}
By Lemma \ref{lem:rbfbar} it is enough to prove that the map
\[
\bar{S}_{\pi,-1}:\zhalf{\rrrrpb{F}}_\chi\to\zhalf{\rredpb{k}}
\]
is an isomorphism.

To prove the result we will show that there is a well-defined $\zhalf{\sgr{F}}$-module homomorphism 
\[
T:\zhalf{\rredpb{k}}\to \zhalf{\rrrrpb{F}}_\chi,\qquad T(\gpb{\bar{a}}) = \gpb{a}_\chi
\]
(where $a\in U$ maps to $\bar{a}\in k^\times$) which is inverse to $\bar{S}_{\pi,-1}$. 

To show that $T$ is well-defined we must show that $\gpb{au}_\chi=\gpb{a}_\chi$ whenever $a\in U$ and 
$u\in U_1$.  

Since clearly 
\[
\bar{S}_{\pi,-1}\circ T=\id{\zhalf{\rrrbl{k}}},  
\]
it is then only necessary to show that $T$ is surjective: we must show that $\gpb{a}_\chi=0$ whenever $v(a)\not=0$ (since 
if $v(a)=0$ then $\gpb{a}_\chi=T(\gpb{\bar{a}})$ by definition of $T$, and thus $\gpb{a}_\chi$ lies in the image of 
$T$ for all $a\in F^\times$).

Note, to begin with, that since $\idem{\chi}=\epm{\pi}\epm{u\pi}$, we have 
\[
\gpb{a}_\chi=0\mbox{ if } a\equiv -\pi,-u\pi\pmod{(F^\times)^2}
\]
by Lemma \ref{cor:ann}.
This is clearly equivalent to
$\gpb{a}_\chi=0$ if $a$ lies in one of the two square classes $\pi\cdot (F^\times)^2$ and $u\pi\cdot (F^\times)^2$.
Thus 
 $\gpb{a}_\chi=0$ if $a \in U\pi^{2n-1}$ for any  $n\in \Z$.

%Next, recall that if $\gpb{a}_\chi=0$ in $\rrrrpb{F}_\chi$ for some $a\in F^\times$, then 
%$\gpb{b}_\chi=0$ for any $b$ lying in the $S_3$-orbit of $a$. 

Thus for any $n\geq 1$, $\gpb{a}_\chi=0$ if $a\in \tau(U\pi^{2n-1})=1-U\pi^{2n-1}\subset U_1$. 

Consider now the case $a=w\pi^{2n}$, $w\in U$ and $n\geq 1$.  By considering the relation
$\left(S_{1/\pi,a/\pi}\right)_\chi$, we see that 
\begin{eqnarray*}
0=
\gpb{\frac{1}{\pi}}_\chi-\gpb{\frac{a}{\pi}}_\chi+\an{\pi}\gpb{a}_\chi
-\an{\pi-1}\gpb{-w\pi^{2n-1}\frac{1-\pi}{1-w\pi^{2n-1}}}_\chi+
\an{-\pi}\gpb{\frac{1}{\pi}\cdot\frac{1-\pi}{1-w\pi^{2n-1}}}_\chi = \an{\pi}\gpb{a}_\chi
\end{eqnarray*} 
since all terms other than the third belong to the square classes $\pi\cdot (F^\times)^2$ or 
$u\pi\cdot (F^\times)^2$.

Thus we deduce that $\gpb{a}_\chi=0$ if $a\in U\pi^{2n}$ for any $n\geq 1$.  It follows that 
$\gpb{a}_\chi=0$ whenever $v(a)>0$. By considering $\sigma(a)=1/a$ it follows that $\gpb{a}_\chi=0$ 
whenever $v(a)\not= 0$. From Lemma \ref{lem:prelim} it follows that $\gpb{a}_\chi=0$ whenever $a\in U_1$ also.  

Finally, let $a\in U\setminus U_1$ and $w\in U_1$. Then
\[
0=\left(S_{a,aw}\right)_\chi=\gpb{a}_\chi-\gpb{aw}_\chi+\an{a}\gpb{w}_\chi-\an{a^{-1}-1}\gpb{w\cdot\frac{1-a}{1-aw}}_\chi+
\an{1-a}\gpb{\frac{1-a}{1-aw}}_\chi
\]
which gives $\gpb{a}_\chi=\gpb{aw}_\chi$ since the last three terms all lie in $U_1$. This proves the Lemma.
\end{proof}

\begin{lem}\label{lem:psi1} 
Let $F$ be a local field whose residue field $k$ has order $q$ with $q\equiv1\pmod{4}$. Then
\[
\zhalf{\rrrrpb{F}}_\psi=\zhalf{\rrrrpb{F}}_{\psi'}=0.
\]
\end{lem}

\begin{proof}
We treat the case of $\psi$. Clearly the case of $\psi'$ is identical since it only involves a switch in our choice of 
uniformizer.

Since $\psi=\epm{u}\epm{\pi}$ and since $-1\in(F^\times)^2$ we have 
\[
\gpb{a}_\psi=0 \mbox{ if } a\equiv u,\pi\pmod{(F^\times)^2}
\]
by Lemma \ref{cor:ann}.

Thus $\gpb{a}_\psi=0$ if $a\in N\pi^{2n}\cup R\pi^{2n-1}$ for any $n\in \Z$.

We must prove that $\gpb{a}_\chi=0$ also for $a\in (F^\times)^2\cup u\pi\cdot (F^\times)^2$.

Taking the $S_3$-action into account as in the last lemma, it follows that 
\[
\gpb{a}_\psi=0 \mbox{ if } a\in 1-R\pi^{2n-1}\cup 1-N\pi^{2n}\mbox{ for any }n\geq 1.
\]

Furthermore, if $a\in R_{-1}$ then $\tau(a)\in u\cdot (F^\times)^2$. Thus $\gpb{\tau(a)}_\psi=0$ and hence $\gpb{a}_\psi=0$.  

Consider now $a=r\pi^{2n}$ with $r\in R$ and $n\geq 1$.  Then we have 
\begin{eqnarray*}
0=\gpb{\frac{\pi}{a}}_\psi-\gpb{\pi}_\psi+\an{\pi}\gpb{a}_\psi
-\gpb{\frac{a-\pi}{1-\pi}}_\psi
+\an{\pi}\gpb{\frac{1}{a}\cdot\frac{a-\pi}{1-\pi}}_\psi
\end{eqnarray*}
which forces $\an{\pi}\gpb{a}_\psi=0=\gpb{a}_\psi$ since the other four terms all belong to the square class 
$\pi\cdot (F^\times)^2$.

By considering $\sigma(a)=a^{-1}$, we deduce that $\gpb{b}_\psi=0$ for all $b\in R\pi^{2n}$ and all $n\not= 0$.

Next, we consider $a=n\pi^{2m-1}$ with $n\in N$ and $m\geq 1$. 
By Lemma \ref{lem:el}, we can write $n=r+s$ with $r,s\in R$. 

Let $x=r\pi^{2m-1}$. Observe that 
\begin{eqnarray*}
w:=&\frac{1-x}{1-a}=\frac{1-r\pi^{2m-1}}{1-n\pi^{2m-1}}
=(1-r\pi^{2m-1})(1+n\pi^{2m-1}+\cdots)\\
=& 1+(n-r)\pi^{2m-1}+\cdots = 1+st\pi^{2m-1} 
\end{eqnarray*}
with $t\in U_1$. Thus $w\in 1-R\pi^{2n-1}$ since $s\in R$.

Thus we have 
\begin{eqnarray*}
0=\gpb{x}_\psi-\gpb{a}_\psi+\an{\pi}\gpb{\frac{n}{r}}_\psi-\an{\pi}\gpb{\frac{n}{r}\cdot w}_\psi+\gpb{w}_\psi=-\gpb{a}_\psi
\end{eqnarray*}
since 
$x\in \pi\cdot (F^\times)^2$, $\frac{n}{r},\frac{nw}{r}\in N\subset u\cdot (F^\times)^2$ and, 
as noted above,  $w\in 1-R\pi^{2m-1}$.

By considering $\sigma(a)$ also, we deduce that $\gpb{b}_\psi=0$ for all $b\in N\pi^{2m-1}$ for any $m\in \Z$.

Thus we have shown that  $\gpb{a}_\psi=0$ whenever $v(a)\not=0$ and whenever $a\in R_{-1}$. 
By Lemma \ref{lem:prelim} it follows that $\gpb{a}_\psi=0$ whenever $a\in U_1$.

It remains to prove that $\gpb{a}_\psi=0$ whenever $a\in R_1$.   

 Let $a\in R_1$. Observe that $\{ as\ |\ s\in R_{-1}\}\cap R_{-1}\not= \emptyset$ since 
$|\bar{R}_{-1}|=1+|\bar{R}_1|$ by Lemma \ref{lem:rn} (1). It follows that we can write $a=t/s$ with $s,t\in R_{-1}$.

Observe that $(1-s)/(1-t)\in R_{-1}$ also: Since $1-s,1-t\in N$ we have $(1-s)/(1-t)\in R$.  Since $a= s/t\in R_1$,
\[
1-\frac{s}{t}=\frac{t-s}{t}\in R
\] 
and thus $s-t,t-s\in R$.   Hence 
\[
1-\frac{1-s}{1-t}=\frac{s-t}{1-t}\in N
\]
so that $(1-s)/(1-t)\in R_{-1}$.  

Applying the same argument to $s^{-1},t^{-1}$ shows that $(1-s^{-1})/(1-t^{-1})\in R_{-1}$ also. Thus 
\[
0=\gpb{s}_\psi-\gpb{t}_\psi+\gpb{a}_\psi-\an{u}\gpb{\frac{1-s^{-1}}{1-t^{-1}}}_\psi+\an{u}\gpb{\frac{1-s}{1-t}}_\psi=\gpb{a}_\psi
\]
proving the lemma.
\end{proof}

\begin{lem}\label{lem:psi3}
Let $F$ be a local field with residue field $k$ of order $q$ and $q\equiv 3\pmod{4}$.
Then 
\[
\zhalf{\rrrrbl{F}}_\psi=\zhalf{\rrrrbl{F}}_{\psi'}=0.
\]
\end{lem}

\begin{proof} By Lemma \ref{lem:mf} it is enough to show that $\lambda_1$ induces  isomorphisms 
of $\zhalf{\sgr{F}}$-modules
\[
\zhalf{\rrrrpb{F}}_\psi\cong M(F)_\psi\mbox{ and }\zhalf{\rrrrpb{F}}_{\psi'}\cong M(F)_{\psi'}.
\]

As remarked in the proof of Lemma \ref{lem:psi1}, it is enough to treat the case of the character $\psi$.

Observe that $M(F)_\psi$ is a free $\zhalf{\Z}$-module of rank $1$ with generator $\idem{\psi}=\epm{-1}\epm{\pi}$. 
Thus, if $L$ is any 
$\zhalf{\sgr{F}}$-module and if $a\in L$, there is a unique $\zhalf{\sgr{F}}$-module homomorphism $M(F)_\psi\to L$ sending
 $\idem{\psi}$ to $\idem{\psi}(a)$.

Thus we fix $n\in N_{-1}$ and let $\Phi:M(F)_\psi\to \zhalf{\rrrrpb{F}}_\psi$ be the 
unique $\zhalf{\sgr{F}}$-module homomorphism sending 
$\idem{\psi}$ to $\frac{1}{4}\gpb{n}_\psi$. 

By our previous remarks, $\Phi$ is well-defined and clearly 
\[
\bar{\lambda}_1\circ\Phi=\id{M(F)_\psi}.
\]

So it only remains to show that $\Phi$ is surjective. 
To do this we show that $\gpb{a}_\psi=0$ in $\zhalf{\rrrrpb{F}}_\psi$ whenever 
$a\not\in N_{-1}$ and that $\gpb{n_1}_\psi=\gpb{n_2}_\psi$ for any $n_1,n_2\in N_{-1}$.    

Now 
\[
\gpb{a}_\psi=0 \mbox{ if } a\in (F^\times)^2\cup -\pi\cdot (F^\times)^2.
\]

In particular, $\gpb{a}_\psi=0$ if $a\in U_1$, since $U_1\subset (F^\times)^2$. By Lemma \ref{lem:prelim} it follows 
that $\gpb{a}_\psi=0$ if $v(a)\not=0$.
% if $v(a)>0$, then $\tau(a)=1-a\in U_1$, and if $v(a)<0$ then $v(\sigma(a))=v(a^{-1})>0$. 

Thus $\gpb{a}=0$ if $a\in \pi\cdot (F^\times)^2$ and $\gpb{a}=0$ if $a\in -1\cdot (F^\times)^2$ and $v(a)\not= 0$. 
All of this shows that 
\[
\gpb{a}_\psi=0\mbox{ for any }a\not\in N
\]
since $N=\{ a\in -1\cdot (F^\times)^2\ |\ v(a)=0\}$.

However, if $a\in N_1$, then $b=\tau(a)\in (F^\times)^2$. Thus $\gpb{b}_\psi=0$ and hence $\gpb{a}_\psi=\gpb{\tau(b)}_\psi=0$.
So we have shown that $\gpb{a}_\psi=0$ for all $a\not\in N_{-1}$.

Finally, let $n_1,n_2\in N_{-1}$. Then in $\zhalf{\rrrrpb{F}}_\psi$ we have 
\begin{eqnarray*}
0=\gpb{n_1}_\psi
-\gpb{n_2}_\psi
+\an{-1}\gpb{\frac{n_2}{n_1}}_\psi
-\an{-1}\gpb{\frac{1-{n_1}^{-1}}{1-{n_2}^{-1}}}_\psi
+\gpb{\frac{1-n_1}{1-n_2}}_\psi=\gpb{n_1}_\psi-\gpb{n_2}_\psi
\end{eqnarray*}
since clearly, from the definition of $N_{-1}$,
\[
\frac{n_2}{n_1},\quad \frac{1-n_1}{1-n_2},\quad \frac{1-{n_1}^{-1}}{1-{n_2}^{-1}}\in R.
\]
This proves the lemma.
\end{proof}

Putting all of these results together gives the following calculation of the third homology of $\spl{2}{F}$:

\begin{thm}\label{thm:local}
Let $F$ be a local field with residue field $k$ of odd order. 
If $\Q_3\subset F$, suppose that $[F:\Q_3]$ is odd.  
 
Then  there is an isomorphism of $\zhalf{\sgr{F}}$-modules
\[
\ho{3}{\spl{2}{F}}{\zhalf{\Z}}\cong \zhalf{\kind{F}}\oplus \zhalf{\pb{k}}
\]
in which $F^\times$ acts trivially on the first factor, while (the square class  of) 
any uniformizer acts as multiplication by $-1$ on the second factor. 

\end{thm}

\begin{proof} 
As observed in section \ref{sec:localprelim} above, there is a $\zhalf{\sgr{F}}$-module decomposition
\[
\ho{3}{\spl{2}{F}}{\zhalf{\Z}}\cong \zhalf{\kind{F}}\oplus\zhalf{\rbl{F}}_{\chi}\oplus\zhalf{\rbl{F}}_{\psi}
\oplus\zhalf{\rbl{F}}_{\psi'}.
\]

Now 
\[
\zhalf{\rbl{F}}_\chi\cong \zhalf{\redpb{k}}=\zhalf{\pb{k}}
\]
by Lemma \ref{lem:rbfbar}, Lemma \ref{lem:mf}  and Lemma \ref{lem:chi}.

Finally, we have 
\[
\zhalf{\rbl{F}}_\psi=\zhalf{\rbl{F}}_{\psi'}=0
\]
by Corollary \ref{cor:psipsi}, Lemma \ref{lem:psi1} and Lemma \ref{lem:psi3}. 

\end{proof}
%\vfill
%\pagebreak

\begin{rem} If $\Q_3\subset F$ and if $[F:\Q_3]$ is even, then the arguments above show that 
there is a surjective homomorphism 
of $\zhalf{\sgr{F}}$-modules   
\[
\xymatrix{
\ho{3}{\spl{2}{F}}{\zhalf{\Z}}\ar@{>>}[r]
& \zhalf{\kind{F}}\oplus \zhalf{\pb{k}}\\
}
\]
whose kernel has order $1$ or $3$.
\end{rem}
\begin{rem} 
It is not difficult to write down an explicit homology class which generates 
$\ho{3}{\spl{2}{F}}{\zhalf{\Z}}_0\cong\zhalf{\pb{\F{q}}}\cong \zhalf{\bl{\F{q}}}$ 
for a local field $F$ with residue field $\F{q}$ of odd order.

Let $r=\zzhalf{(q+1)}$. 
Let $a\in U\setminus U^2$ and observe that the 
 residue field of $F(\sqrt{a})$ is $\F{q}(\sqrt{\bar{a}})=\F{q^2}$. Let $z$ be an 
element of $\F{q}(\sqrt{\bar{a}})^\times$ of order $r$. By Hensel's Lemma, there exists $\zeta\in F(\sqrt{a})^\times$ 
of order 
$r$ with the property that $\bar{\zeta}=z$; i.e. $\zeta$ is a primitive $r$-th root of unity. Thus, if 
$\zeta=\alpha+\beta\sqrt{a}$, then $z=\bar{\alpha}+\bar{\beta}\sqrt{\bar{a}}$.

Now, for any quadratic field extension $K/L$ with $K=L(\sqrt{b})$ for some $b\in L^\times$, there is an injective 
group homomorphism 
\[
\kappa: K^\times \to \gl{2}{L}, x+y\sqrt{b}\mapsto \matr{x}{by}{y}{x}
\] 
with the property that $\det(\kappa(x))=N_{K/L}(x)\mbox{ for all } x\in K^\times$.

Let
\[
t=\kappa(\zeta)=\matr{\alpha}{a\beta}{\beta}{\alpha}\in\spl{2}{F}.
\]
 Let $\Gamma_t$ be a generator of 
\[
\ho{3}{\an{t}}{\zhalf{\Z}}\cong \Z/r.
\]

Then the image, $\gamma$,  of $\Gamma_t$ under the map 
\[
\ho{3}{\an{t}}{\zhalf{\Z}}\to\ho{3}{\spl{2}{F}}{\zhalf{\Z}}
\]
is represented by the cycle  
\[
\sum_{i=0}^{n-1}1\otimes(1,t,t^i,t^{i+1}).
\]

Let $\bar{t}\in \spl{2}{\F{q}}$ be the matrix obtained by reducing the entries of $t$. Then $\bar{t}=\kappa(z)$ 
has order $r$ by Lemma \ref{lem:arxivcplx}. Thus the  class 
$\gamma$ maps to a generator of $\zhalf{\bl{\F{q}}}$ under the specialization homomorphism $S_{\pi,-1}$ since the diagram 

\[
\xymatrix{
\ho{3}{\an{t}}{\zhalf{\Z}}\ar[r]\ar[d]^-{\cong }
&
\ho{3}{\spl{2}{F}}{\zhalf{\Z}}\ar[r]
&
\zhalf{\rbl{F}}\ar[d]^-{S_{\pi,-1}}
\\
\ho{3}{\an{\bar{t}}}{\zhalf{\Z}}\ar[r]
&
\ho{3}{\spl{2}{\F{q}}}{\zhalf{\Z}}\ar[r]
&
\zhalf{\bl{\F{q}}}
}
\]
commutes, because the composite map on the bottom row is an isomorphism (see Lemma \ref{lem:arxivcplx} above).

Let $\pi$ be a uniformizing parameter and let 
\[
t_\pi=
\matr{\pi}{0}{0}{1}\matr{\alpha}{a\beta}{\beta}{\alpha}\matr{1/\pi}{0}{0}{1}=\matr{\alpha}{\pi a\beta}{\beta/\pi}{\alpha}.
\] 
Then the homology class
\[
\epm{\pi}(\gamma)=\frac{1}{2}\sum_{i=0}^{n-1}1\otimes\left((1,t,t^i,t^{i+1})-(1,t_\pi,t_\pi^i,t_\pi^{i+1})\right)
\]
 is a generator of $\ho{3}{\spl{2}{F}}{\zhalf{\Z}}_0=\epm{\pi}\ho{3}{\spl{2}{F}}{\zhalf{\Z}}$ 
by the arguments of this section.
\end{rem}

\begin{rem} The same techniques as used above apply also to case where the residue field has characteristic $2$. However,
the number of square classes is of the form $2^k$, $k\geq 3$ and the condition $U_1=U_1^2$ is not 
satisfied. The `$S_3$-dynamics' -- i.e. the manner in which the $S_3$ orbits intersect with the 
square classes in $F^\times$ -- become  more complicated as $k$ grows. 
The author has confirmed, for example, that 
\[
\ho{3}{\spl{2}{\Q_2}}{\zhalf{\Z}}\cong\zhalf{\kind{\Q_2}}\oplus\pb{\F{2}}
\]
as expected, but the calculations are long and unedifying. 
\end{rem}

\bibliography{blochgroup}

\def\cprime{$'$}
\begin{thebibliography}{10}

\bibitem{adem:naffah}
Alejandro Adem and Nadim Naffah.
\newblock On the cohomology of {${\rm SL}\sb 2(\bold Z[1/p])$}.
\newblock In {\em Geometry and cohomology in group theory ({D}urham, 1994)},
  volume 252 of {\em London Math. Soc. Lecture Note Ser.}, pages 1--9.
  Cambridge Univ. Press, Cambridge, 1998.

\bibitem{bloch:regulators}
Spencer~J. Bloch.
\newblock {\em Higher regulators, algebraic {$K$}-theory, and zeta functions of
  elliptic curves}, volume~11 of {\em CRM Monograph Series}.
\newblock American Mathematical Society, Providence, RI, 2000.

\bibitem{borel:serre}
A.~Borel and J.-P. Serre.
\newblock Cohomologie d'immeubles et de groupes {$S$}-arithm\'etiques.
\newblock {\em Topology}, 15(3):211--232, 1976.

\bibitem{sah:dupont}
Johan~L. Dupont and Chih~Han Sah.
\newblock Scissors congruences. {II}.
\newblock {\em J. Pure Appl. Algebra}, 25(2):159--195, 1982.

\bibitem{zickert:goette}
Sebastian Goette and Christian~K. Zickert.
\newblock The extended {B}loch group and the {C}heeger-{C}hern-{S}imons class.
\newblock {\em Geom. Topol.}, 11:1623--1635, 2007.

\bibitem{hesselmann}
Sabine Hesselmann.
\newblock {\em Zur {T}orsion der {K}ohomologie {$S$}-arithmetischer {G}ruppen}.
\newblock Bonner Mathematische Schriften [Bonn Mathematical Publications], 257.
  Universit\"at Bonn Mathematisches Institut, Bonn, 1993.

\bibitem{hut:arxivcplx11}
Kevin Hutchinson.
\newblock {A Bloch-Wigner complex for $\mathrm{SL}_2$}.
\newblock To appear in Journal of $K$-theory.

\bibitem{hutchinson:tao2}
Kevin Hutchinson and Liqun Tao.
\newblock The third homology of the special linear group of a field.
\newblock {\em J. Pure Appl. Algebra}, 213:1665--1680, 2009.

\bibitem{hutchinson:tao3}
Kevin Hutchinson and Liqun Tao.
\newblock Homology stability for the special linear group of a field and
  {M}ilnor-{W}itt ${K}$-theory.
\newblock {\em Doc. Math.}, (Extra Vol.):267--315, 2010.

\bibitem{knudson:book}
Kevin~P. Knudson.
\newblock {\em Homology of linear groups}, volume 193 of {\em Progress in
  Mathematics}.
\newblock Birkh\"auser Verlag, Basel, 2001.

\bibitem{kuhnlein}
Stefan K{\"u}hnlein.
\newblock On torsion in the cohomology of {${\bf PSL}\sb 2(\Bbb Z[1/N])$}.
\newblock {\em J. Number Theory}, 99(2):284--297, 2003.

\bibitem{mirzaii:third}
B.~Mirzaii.
\newblock Third homology of general linear groups.
\newblock {\em J. Algebra}, 320(5):1851--1877, 2008.

\bibitem{nahm:bloch}
Werner Nahm.
\newblock Conformal field theory and torsion elements of the {B}loch group.
\newblock In {\em Frontiers in number theory, physics, and geometry. {II}},
  pages 67--132. Springer, Berlin, 2007.

\bibitem{neukirch:ant}
J{\"u}rgen Neukirch.
\newblock {\em Algebraic number theory}, volume 322 of {\em Grundlehren der
  Mathematischen Wissenschaften [Fundamental Principles of Mathematical
  Sciences]}.
\newblock Springer-Verlag, Berlin, 1999.
\newblock Translated from the 1992 German original and with a note by Norbert
  Schappacher, With a foreword by G. Harder.

\bibitem{neumann:yang}
Walter~D. Neumann and Jun Yang.
\newblock Bloch invariants of hyperbolic {$3$}-manifolds.
\newblock {\em Duke Math. J.}, 96(1):29--59, 1999.

\bibitem{neumann:zagier}
Walter~D. Neumann and Don Zagier.
\newblock Volumes of hyperbolic three-manifolds.
\newblock {\em Topology}, 24(3):307--332, 1985.

\bibitem{rahm:bianchi}
Alexander~D. Rahm.
\newblock Homology and {$K$}-theory of the {B}ianchi groups.
\newblock {\em C. R. Math. Acad. Sci. Paris}, 349(11-12):615--619, 2011.

\bibitem{rahm:fuchs}
Alexander~D. Rahm and Mathias Fuchs.
\newblock The integral homology of {${\rm PSL}\sb 2$} of imaginary quadratic
  integers with nontrivial class group.
\newblock {\em J. Pure Appl. Algebra}, 215(6):1443--1472, 2011.

\bibitem{sah:discrete3}
Chih-Han Sah.
\newblock Homology of classical {L}ie groups made discrete. {III}.
\newblock {\em J. Pure Appl. Algebra}, 56(3):269--312, 1989.

\bibitem{sengun:bianchi}
Mehmet~Haluk {\c{S}}eng{\"u}n.
\newblock On the integral cohomology of {B}ianchi groups.
\newblock {\em Exp. Math.}, 20(4):487--505, 2011.

\bibitem{sus:bloch}
A.~A. Suslin.
\newblock {$K\sb 3$} of a field, and the {B}loch group.
\newblock {\em Trudy Mat. Inst. Steklov.}, 183:180--199, 229, 1990.
\newblock Translated in Proc.\ Steklov Inst.\ Math.\ {\bf 1991}, no.\ 4,
  217--239, Galois theory, rings, algebraic groups and their applications
  (Russian).

\bibitem{zagier:dilog}
Don Zagier.
\newblock The dilogarithm function.
\newblock In {\em Frontiers in number theory, physics, and geometry. {II}},
  pages 3--65. Springer, Berlin, 2007.

\end{thebibliography}
\end{document}